\documentclass{article}
\usepackage[margin=1.15in]{geometry}
\usepackage{amsmath,amssymb,amsthm,amsfonts}
\usepackage{graphicx,tikz}
\usetikzlibrary{calc,patterns}
\usepackage{enumerate}
\usepackage{subcaption}
\usepackage[ruled,vlined]{algorithm2e}

\usepackage{listings}
\lstdefinelanguage{Sage}[]{Python}
{morekeywords={False,sage,True},sensitive=true}
\definecolor{dblackcolor}{rgb}{0.0,0.0,0.0}
\definecolor{dbluecolor}{rgb}{0.01,0.02,0.7}
\definecolor{dgreencolor}{rgb}{0.2,0.4,0.0}
\definecolor{dgraycolor}{rgb}{0.30,0.3,0.30}

\newcommand{\RR}{\mathbb{R}}
\newcommand{\ZZ}{\mathbb{Z}}
\newcommand{\HH}{\mathcal{H}}
\newcommand{\FF}{\mathcal{F}}
\newcommand{\LL}{\mathcal{L}}
\newcommand{\bt}{\mathbf{t}}

\DeclareMathOperator{\conv}{conv}
\DeclareMathOperator{\Ad}{Ad}

\newtheorem{theorem}{Theorem}[section]
\newtheorem{corollary}[theorem]{Corollary}
\newtheorem{prop}[theorem]{Proposition}

\theoremstyle{definition}
\newtheorem{defn}[theorem]{Definition}

\theoremstyle{remark}
\newtheorem{remark}[theorem]{Remark}

\theoremstyle{example}
\newtheorem{example}[theorem]{Example}

\usepackage{hyperref}
\renewcommand{\d}{\partial}

\newcommand{\pd}[1]{}

\newcommand{\Addresses}{{% additional braces for segregating \footnotesize
  \bigskip
  \footnotesize

  \textsc{Technische Universit\"at Berlin, Institut f\"ur Mathematik, Stra{\ss}e des 17. Juni 136, 10623 Berlin}\par\nopagebreak
  \textit{E-mail address}: \texttt{buesra.sert@tu-berlin.de}

  \medskip

  \textsc{Fachbereich Mathematik und Informatik, Freie Universit\"at Berlin, Arnimallee 14, 14195 Berlin }\par\nopagebreak
  \textit{E-mail address}: \texttt{niklas.livchitz@gmail.com}

  \medskip

  \textsc{Discrete Geometry Group, Freie Universit\"at Berlin, Arnimallee 2, 14195 Berlin, Germany}\par\nopagebreak
  \textit{E-mail address}: \texttt{w.amy.math@gmail.com}

}}

\title{Combinatorial Methods for Minkowski Tensors of Polytopes}

\author{Niklas Livchitz, B\"u\c{s}ra Sert and Amy Wiebe }
\date{\today}

\begin{document}

\allowdisplaybreaks

\maketitle

\abstract{In this paper we use a generating function approach to record and calculate entries of the Minkowski tensors of a polytope. We focus on ``surface tensors'', extending the methods used in \cite{momentpaper} for moments of the uniform distribution which correspond to volume tensors. In this context we also extend the definition of the adjoint polynomial to the boundary complex of a polytope with simplicial facets. In the case of simplicial polytopes we give an explicit formulation for these surface tensors. 
}

\section{Introduction}

Minkowski tensors were first introduced by McMullen \cite{McMullen_1997} in the context of valuations on convex bodies. They are of theoretical interest in this setting, since it was shown by Alesker \cite{Alesker2, Alesker1} that when multiplied with powers of the metric tensor they span the vector space of isometry covariant, continuous, tensor-valued valutions on the set of convex bodies. Minkowski tensors have also found importance in applications, as they can be used to describe more complex geometric quantities -- beyond scalar values such as volume and surface area -- including shape, orientation, and anisotropy \cite{Aniso}. Minkowski tensors have appeared in the study of such subjects as material sciences \cite{Schroder_Turk_2011}, biosciences \cite{Beisbart_2006}, and digital image analysis \cite{Voronoi}. 

There is also a close connection between these tensors and  moments  of  the  uniform  distribution  on  convex  bodies \cite{Thesis}.  As shown in \cite{momentpaper}, moments on polytopes can be described via rational generating functions and are parametrized by certain algebraic  varieties.   
In this paper we show that the methods employed in \cite{momentpaper} for describing moments,  and hence volume  tensors,  extend to give  a  description  of  surface  tensors. We  compute  explicit   generating  functions for these tensors and  define  an  analog  of  the adjoint polynomial of \cite{Warren_1996} for the boundary complex of a polytope.  Using the generating function we also describe methods for deriving explicit tensors. In particular, we use the method of partial derivatives to extract coefficients, and we note that the especially simple form of the generating function for simplicial polytopes allows us to give a complete description of their surface tensors. 

The remainder of the paper is organized as follows. In Section~\ref{sec:bg} we introduce the necessary background on Minkowski tensors, giving both the general definition for convex bodies and a simplified form for polytopes. In Section~\ref{sec:moments} we describe the connection of Minkowski tensors to moments of a probability measure on a polytope. Here we present results from \cite{momentpaper} on moment generating functions and the adjoint polynomial of a polytope. In Section~\ref{sec:genfn} we extend the methods of the previous section to give a generating function for the Minkowski ``surface tensors'' of a polytope~$P$ which has a closed form rational expression. From this expression, we introduce the {\em surface adjoint} of $P$, which is a generalization of the adjoint polynomial of $P$ to its boundary complex, and discuss its vanishing behavior. In Section~\ref{sec:simplicial}, we extract coefficients of the generating function  and  give  an explicit  description  of  the  surface  tensors  in  terms  of  partial derivatives.  In this section we also give a collection of examples. {In Section~\ref{sec:q} we end with some questions for further exploration. In Appendix~\ref{sec:App} one can find an algorithmic description of our generating function, and an example of how to apply the results of Section~\ref{sec:simplicial} using Sage.}

The authors thank Rainer Sinn for his helpful discussion on early versions of this paper. {They also acknowledge support from the {\em Thematic Einstein Semester on
Algebraic Geometry:
Varieties, Polyhedra, Computation} for the research retreat which facilitated this collaboration.}

\section{Background} \label{sec:bg}

Let $K$ be a convex, compact subset of $\RR^d$ {with nonempty interior} (i.e., a {\em convex body}). We begin by recalling the definition of the intrinsic volumes of $K$. 

\begin{defn}The {\em intrinsic volumes} $V_0(K),\ldots, V_d(K)$ of a convex body are defined by the coefficients of the polynomial in $\RR[\varepsilon]$ on the right hand side of the following
\begin{equation}\lambda(K+\varepsilon B^d) = \sum_{j=0}^d \varepsilon^{d-j}\kappa_{d-j}V_j(K), \label{EQ:Steiner} \end{equation}
where $\lambda$ denotes Lebesgue measure on $\RR^d$ and $\kappa_{d-j}$ is the volume of the unit ball $B^{d-j}$ in $\RR^{d-j}$. 
\end{defn}
Equation \eqref{EQ:Steiner} is known as the Steiner formula. It was proven that the volume of the parallel set (the left hand side above) is a polynomial in $\varepsilon>0$ by Steiner \cite{Steiner} for polytopes and sufficiently smooth surfaces in $\RR^2$ and $\RR^3$. 

The term {\em intrinsic volumes} was coined by McMullen \cite{McMullen_1975} and comes from the fact that the $V_i(K)$ depend only on $K$ and not on the ambient space. In particular, $V_m(K)$ is the $m$-dimensional volume of $K$ if $K$ is contained in an affine subspace of $\RR^d$ of dimension~$m$. So if $K$ is a $d$-dimensional convex body, then $V_d$ is its volume, and moreover, $2V_{d-1}$ is its surface area, and $V_0$ is its Euler characteristic. The intrinsic volumes (up to normalization) are also known as {\em querma\ss integrals} and {\em Minkowski functionals}.  

We will see that the Minkowski tensors are a tensor-valued extension of these scalar-valued instrinsic volumes or Minkowski functionals. That is, a Minkowski tensor $\Phi_j^{r,s}$ will be a map that takes a convex body in $\RR^d$ to a rank $r+s$ tensor over $\RR^d$ such that it coincidences with the intrinsic volume $V_j$ when the rank is $0$. 

Given a symmetric tensor $T$ of rank $r$ over $\RR^d$ (i.e., a symmetric multilinear function of $r$ variables in $\RR^d$), we will identify it (by multilinearity) with the array
$$\{T(e_{i_1},\ldots,e_{i_r})\}_{i_1,\ldots,i_r=1}^d$$
of its components, where $\{e_i\}_{i=1}^d$ is the standard basis of $\RR^d$. 
Now identifying $\RR^d$ with its dual via the inner product (so that $x\in\RR^d$ is a linear functional on $\RR^d$), we write $x^r$ for the $r$-fold symmetric tensor of $x\in\RR^d$. We then identify the tensor $x^r$ with its array of components
$$\{x_{i_1}\cdots x_{i_r}\}_{i_1,\ldots, i_r=1}^d.$$

\begin{remark}We use the convention common in the literature on Minkowski tensors to refer to the number of inputs $r$ to a tensor as its {\em rank}. This value is also sometimes refered to as the {\em order} of a tensor, and it is important to note that this is different from the rank defined as the minimum $n$ such that the tensor can be written as a sum of $n$ simple tensors. Under the latter definition, $x^r$ is always rank $1$ for $x\neq 0\in\RR^d$. 
\end{remark}

\begin{example} If $x = (x_1,x_2,x_3)\in\RR^3$, then the rank $2$ tensor $x^2$ can be written
\begin{align*}
x^2 := x\otimes x & = (x_1e_1+x_2e_2+x_3e_3)\otimes(x_1e_1+x_2e_2+x_3e_3) \\
& \equiv \left[\begin{array}{ccc} 	
			x_1x_1 & x_1x_2 & x_1x_3 \\
			x_2x_1 & x_2x_2 & x_2x_3 \\
			x_3x_1 & x_3x_2 & x_3x_3	\end{array}\right].
\end{align*}
\end{example}

Now to define the desired tensor-valued functions we start with
 a local version of the Steiner formula~\eqref{EQ:Steiner}. Given a point $x\in\RR^d$, let $p(K,x)$ denote the (unique) nearest point to $x$ in $K$. Then 
$$u(K,x) = \frac{x-p(K,x)}{||x-p(K,x)||}$$
is the unit outer normal of $K$ at $p(K,x)$. Let $S^{d-1}$ be the unit sphere in $\RR^d$. For $\varepsilon>0$ and Borel set $A = \beta\times\omega\subset\RR^d\times S^{d-1}$, the volume of the local parallel set 
$$M_\varepsilon(K,A) = \left\{x\in(K+\varepsilon B^d)\setminus K : \left(p(K,x),u(K,x)\right)\in A\right\}$$
is a polynomial in $\varepsilon$ of degree at most $d-1$; that is,
\begin{equation}M_\varepsilon(K,A) = \sum_{j=0}^{d-1} \varepsilon^{d-j}\kappa_{d-j} \Lambda_j(K,A).\label{EQ:local_Steiner}\end{equation}

\begin{defn} The coefficients $\Lambda_j(K,\cdot)$ in \eqref{EQ:local_Steiner} above define the {\em support measures} of~$K$. Further, define $\Lambda_j(\emptyset,\cdot) = 0$ for $j=0,\ldots, d$.
\end{defn}

\begin{figure}[ht]
\begin{subfigure}{.4\textwidth}

 \begin{center}
\begin{tikzpicture}[xscale=1.2, yscale=1.2]
  \coordinate (p1) at (0,0) {};
  \coordinate (p2) at (1,2) {};
  \coordinate (p3) at (3,2) {};
  \coordinate (p4) at (4,0) {};
  \coordinate (p5) at (0.3,0) {};
  \coordinate (p6) at (-0.3,0) {};
  \coordinate (p7) at (1.3,2) {};
  \coordinate (p8) at (0.7,2) {};
  \coordinate (p9) at (3.3,2) {};
  \coordinate (p10) at (2.7,2) {};
  \coordinate (p11) at (4.3,0) {};
  \coordinate (p12) at (3.7,0) {};
  \coordinate (p13) at (0,0.3) {};
  \coordinate (p14) at (0,-0.3) {};
  \coordinate (p15) at (4,0.3) {};
  \coordinate (p16) at (4,-0.3) {};
  \coordinate (p17) at (3,2.3) {};
  \coordinate (p18) at (3,1.7) {};
  \coordinate (p19) at (1,2.3) {};
  \coordinate (p20) at (1,1.7) {};
  \coordinate (p21) at (0.3,0.3) {};
  \coordinate (p22) at (1.3,1.7) {};
  \coordinate (p23) at (3.3,1.7) {};
  \coordinate (p24) at (3.7,0.3) {};
  \coordinate (p25) at (-0.3,0.3) {};
  \coordinate (p26) at (0.7,1.7) {};
  \coordinate (p27) at (2.7,1.7) {};
  \coordinate (p28) at (4.3,0.3) {};

  \draw [ fill=pink, color=pink, draw=gray, thick] (p1) -- (p2) -- (p3) -- (p4) -- cycle;
  \draw [fill=yellow, color=yellow, opacity=0.5] (p1) circle [radius=0.3];
  \draw [fill=yellow, color=yellow,opacity=0.5] (p2) circle [radius=0.3];
  \draw [fill=yellow, color=yellow,opacity=0.5] (p4) circle [radius=0.3];
  \draw[fill=yellow, color=yellow,opacity=0.5] (p6) -- (p5) -- (p7) -- (p8) -- (p6) ;
  \draw[fill=yellow, color=yellow,opacity=0.5] (p10) -- (p9) -- (p11) -- (p12) -- (p10) ;
  \draw[fill=yellow, color=yellow,opacity=0.5] (p13) -- (p14) -- (p16) -- (p15) -- (p13) ;
  \draw[fill=yellow, color=yellow,opacity=0.5] (p19) -- (p20) -- (p18) -- (p17) -- (p19) ;
  \draw [fill=yellow, color=yellow, draw=orange, thick, opacity=0.7] (p3) circle [radius=0.3];
  \draw [color=orange, thick,opacity=0.5] (p3) -- (p17);
  \node (e) [left] at ($ (p3)!0.28!(p17)$) {$\epsilon$};
  \draw [draw=orange, thick, opacity=0.7] (2,0) circle [radius=0.3];
  \node (ee) [left] at (2.1, -0.2) {$\epsilon$};
  \draw [color=orange, thick,opacity=0.5] (2,0) -- (2,-0.3);
  \node (P) at (2,1) {$P$};

\end{tikzpicture}
\end{center}

  \caption{Parallel set}
  \label{FIG:sub-parallel}
\end{subfigure}
\begin{subfigure}{.6\textwidth}

 \begin{center}
\begin{tikzpicture}[xscale=1.1, yscale=1.1]
  \coordinate (p1) at (0,0) {};
  \coordinate (p2) at (1,2) {};
  \coordinate (p3) at (3,2) {};
  \coordinate (p4) at (4,0) {};
  \coordinate (p5) at (0.3,0) {};
  \coordinate (p6) at (-0.3,0) {};
  \coordinate (p7) at (1.3,2) {};
  \coordinate (p8) at (0.7,2) {};
  \coordinate (p9) at (3.3,2) {};
  \coordinate (p10) at (2.7,2) {};
  \coordinate (p11) at (4.3,0) {};
  \coordinate (p12) at (3.7,0) {};
  \coordinate (p13) at (0,0.3) {};
  \coordinate (p14) at (0,-0.3) {};
  \coordinate (p15) at (4,0.3) {};
  \coordinate (p16) at (4,-0.3) {};
  \coordinate (p17) at (3,2.3) {};
  \coordinate (p18) at (3,1.7) {};
  \coordinate (p19) at (1,2.3) {};
  \coordinate (p20) at (1,1.7) {};
  \coordinate (p21) at (0.3,0.3) {};
  \coordinate (p22) at (1.3,1.7) {};
  \coordinate (p23) at (3.3,1.7) {};
  \coordinate (p24) at (3.7,0.3) {};
  \coordinate (p25) at (-0.3,0.3) {};
  \coordinate (p26) at (0.7,1.7) {};
  \coordinate (p27) at (2.7,1.7) {};
  \coordinate (p28) at (4.3,0.3) {};

  \draw [ fill=pink, color=pink, draw=gray, thick] (p1) -- (p2) -- (p3) -- (p4) -- cycle;
  \path[fill=orange, opacity=0.3, rounded corners=20pt ] (-1,4) -- (2,3.5) --(4,3.8) -- (3.3,2.5) -- (3,0.4) -- (2,1.3) -- (0, 1.7) -- cycle;
  \node (B) at (0,3) {$\beta$};
  
  \draw [] (5,1)--(5,3);
  \draw[] (4,2) --(6,2);
  \node(o) at (5,2){};
  \draw[thick, color=magenta] ($(o)+(0,0.7)$) arc[radius=0.7, start angle=90, end angle=180];
  \draw[thick, color=black] ($(o)+(0,0.7)$) arc[radius=0.7, start angle=90, end angle=-180];
  \node[color=magenta](w) at (4.2,2.5) {$\omega$};
   \node[color=black](s) at (6,2.8) {$S^{1}$};
\draw[fill=yellow, color=yellow,opacity=0.7] (p2) -- ($(p2)+(-0.42,0.02)$) -- ($(p1)+(0.37,1.67)$) -- ($(p1)+(0.75,1.57)$) -- (p2) ;

  \draw[fill=yellow, color=yellow,opacity=0.7] (p2) -- ($(p2)+(0,0.4)$) -- ($(p3)+(0,0.4)$) -- (p3) -- (p2) ;
  \filldraw[thick, color=yellow] ($(p2)+(0,0.4)$) arc[radius=0.4, start angle=90, end angle=180] -- (p8) -- (p2);
  \node [] (m) at (2, 2.2) {${M_{\epsilon}(P,\beta\times\omega)}$};
  
  \node (P) at (2,1) {$P$};
\end{tikzpicture}
\end{center}

  \caption{Local parallel set for $A=\beta\times \omega$}
  \label{FIG:sub-local}
\end{subfigure}
\caption{Parallel and local parallel sets for a polytope $P$.}
\label{FIG:parallel_sets}
\end{figure}
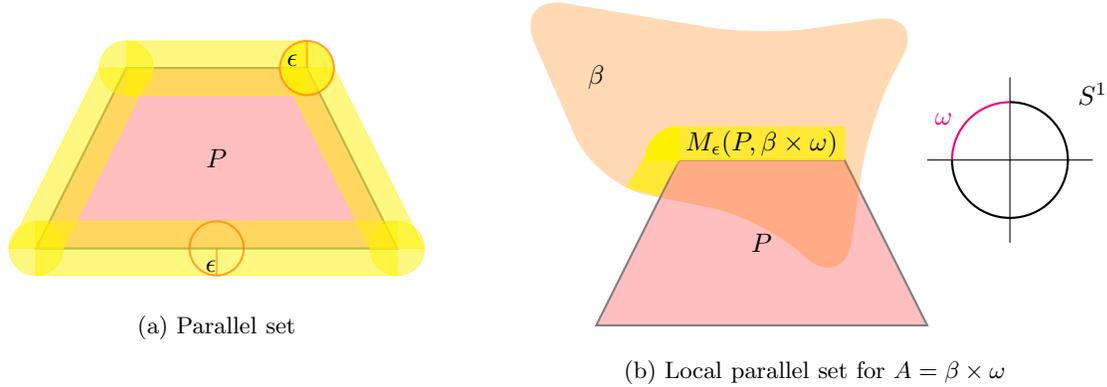

Notice that if $A=\RR^d\times S^{d-1}$, then $\Lambda_j(K,A) = V_j(K)$. Furthermore, in this case, the points $p(K,x),u(K,x)\in A$ will range over the entire {\em normal cycle} of $K$; that is, the set of all pairs $(p,u)$ such that $p$ is a point on the boundary of $K$ and $u$ is a unit normal to $K$ at $p$. 

We 
are now able to define 
the main objects we are interested in, the Minkowski tensors.

\begin{defn}
For a $d$-dimensional convex body $K$, and integers $r,s\geq 0$, the {\em Minkowski tensors} of $K$ are 
\begin{equation}
\Phi_j^{r,s}(K) = \frac{\omega_{d-j}}{r!s!\,\omega_{d-j+s}}\int_{\RR^d\times S^{d-1}} x^ru^s \Lambda_j(K,d(x,u))
\label{EQ:the_tensors} \end{equation}
for $j=0,\ldots,d-1$, where $\omega_j$ is the surface area of the unit ball in $\RR^j$, and
\begin{equation}
\Phi_d^{r,0}(K) = \frac{1}{r!}\int_K x^r\lambda(dx).
\label{EQ:vol_tensor} \end{equation}
For other choices of $r,s,j$ define $\Phi_j^{r,s}=0$. 
The tensors in \eqref{EQ:vol_tensor} are called {\em volume tensors}. For $j=d-1$ and $r=0$, the tensors in \eqref{EQ:the_tensors} are called {\em surface tensors}. 
\end{defn}

Recall that since $x^r = x^{\otimes r}$ and $u^s = u^{\otimes s}$, $\Phi^{r,s}_j(K)$ is a tensor of rank $r+s$. 
Then notice that for $j=0,\ldots,d$, we have the rank $0$ tensor $\Phi_j^{0,0} = V_j$, so that these tensors are extensions of the intrinsic volume as noted earlier. Other volume tensors also have well-known physical interpretations. For instance, $\Phi^{1,0}_d$, after normalization by the volume of $K$ is the center of gravity, and $\Phi^{2,0}_d$ is  related to the tensor of inertia (see for example, \cite[Section 1.3]{Thesis},\cite{Aniso}).

Restricting to the case of polytopes, $K=P$, it is not hard to see that the surface measures will depend on the faces of $P$ and their normal cones. Let $\FF_j(P)$ denote the set of $j$-dimensional faces of the polytope $P$, and for any face $F$ of $P$, denote by $N(F,P)$, the outer normal cone of $F$.  (See Figure~\ref{FIG:normals}.)

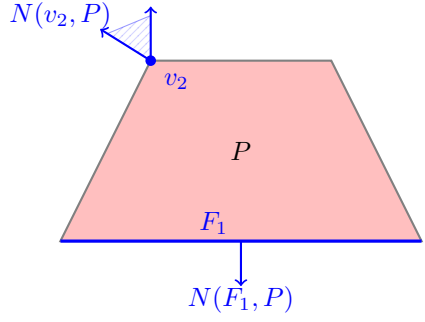
\begin{figure}
\centering
\begin{tikzpicture}[xscale=1.2, yscale=1.2]
  \coordinate (p1) at (0,0) {};
  \coordinate (p2) at (1,2) {};
  \coordinate (p3) at (3,2) {};
  \coordinate (p4) at (4,0) {};
  \coordinate (p5) at (0.3,0) {};
  \coordinate (p6) at (-0.3,0) {};
  \coordinate (p7) at (1.3,2) {};
  \coordinate (p8) at (0.7,2) {};
  \coordinate (p9) at (3.3,2) {};
  \coordinate (p10) at (2.7,2) {};
  \coordinate (p11) at (4.3,0) {};
  \coordinate (p12) at (3.7,0) {};
  \coordinate (p13) at (0,0.3) {};
  \coordinate (p14) at (0,-0.3) {};
  \coordinate (p15) at (4,0.3) {};
  \coordinate (p16) at (4,-0.3) {};
  \coordinate (p17) at (3,2.3) {};
  \coordinate (p18) at (3,1.7) {};
  \coordinate (p19) at (1,2.3) {};
  \coordinate (p20) at (1,1.7) {};
  \coordinate (p21) at (0.3,0.3) {};
  \coordinate (p22) at (1.3,1.7) {};
  \coordinate (p23) at (3.3,1.7) {};
  \coordinate (p24) at (3.7,0.3) {};
  \coordinate (p25) at (-0.3,0.3) {};
  \coordinate (p26) at (0.7,1.7) {};
  \coordinate (p27) at (2.7,1.7) {};
  \coordinate (p28) at (4.3,0.3) {};
  
  \draw [ fill=pink, color=pink, draw=gray, thick] (p1) -- (p2) -- (p3) -- (p4) -- cycle;

     \node (P) at (2,1) {$P$};
     \draw[->,blue, thick] (p2) -- ($(p2)+(-0.56,0.35)$);
     \draw[->,blue, thick] (p2) -- ($(p2)+(0,0.6)$);
     \draw[blue, pattern=north east lines, pattern color=blue, opacity=0.4] ($(p2)+(0,0.5)$) -- ($(p2)+(-0.48,0.3)$) -- (p2) -- cycle;
     \node [color=blue] at ($(p2)+(-1,0.5)$) {$N(v_2,P)$};
     \node [fill, color=blue, shape=circle, minimum size = 4pt, inner sep = 0pt, label={-45}:\textcolor{blue}{$v_2$}] at (p2) {};
     \draw[->, thick, blue] (2,0) -- (2,-0.5);
     \draw[very thick, blue] (p1) -- (p4);
     \node [color=blue] at (1.7, 0.2) {$F_1$};
     \node [color=blue] at (2,-0.65) {$N(F_1,P)$};
\end{tikzpicture}
\caption{Normal cones of faces of a polytope.}
\label{FIG:normals}
\end{figure}

Then we get the following description of the surface measures of a polytope
\begin{equation}
\Lambda_j(P,A) = \frac{1}{\omega_{d-j}}\sum_{F\in\FF_j(P)}\int_F\int_{N(F,P)\cap S^{d-1}} \mathbf{1}_A(x,u)\HH^{d-j-1}(du)\HH^j(dx),
\label{EQ:polytope_measures}\end{equation}
where $\HH^j$ is $j$-dimensional Hausdorff measure and $\mathbf{1}_A$ is the indicator function for set~$A$. 

{
\begin{example} 
One can gain some intuition regarding \eqref{EQ:polytope_measures} by picturing the following example:
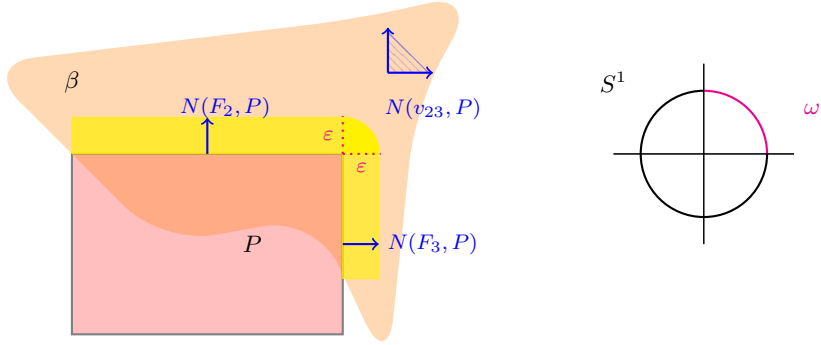
\begin{center} 
\begin{tikzpicture}[xscale=1.2, yscale=1.2]
  \coordinate (p1) at (0,0) {};
  \coordinate (p2) at (0,2) {};
  \coordinate (p3) at (3,2) {};
  \coordinate (p4) at (3,0) {};
  \coordinate (p5) at (0.3,0) {};
  \coordinate (p6) at (-0.3,0) {};
  \coordinate (p7) at (0.3,2) {};
  \coordinate (p8) at (0.7,2) {};
  \coordinate (p9) at (3.3,2) {};
  \coordinate (p10) at (2.7,2) {};
  \coordinate (p11) at (3.3,0) {};
  \coordinate (p12) at (3.7,0) {};
  \coordinate (p13) at (0,0.3) {};
  \coordinate (p14) at (0,-0.3) {};
  \coordinate (p15) at (3,0.3) {};
  \coordinate (p16) at (3,-0.3) {};
  \coordinate (p17) at (3,2.3) {};
  \coordinate (p18) at (3,1.7) {};
  \coordinate (p19) at (0,2.3) {};
  \coordinate (p20) at (0,1.7) {};
  \coordinate (p21) at (0.3,0.3) {};
  \coordinate (p22) at (1.3,1.7) {};
  \coordinate (p23) at (2.3,1.7) {};
  \coordinate (p24) at (2.7,0.3) {};
  \coordinate (p25) at (-0.3,0.3) {};
  \coordinate (p26) at (0.7,1.7) {};
  \coordinate (p27) at (2.7,1.7) {};
  \coordinate (p28) at (3.3,0.3) {};
  \draw [ fill=pink, color=pink, draw=gray, thick] (p1) -- (p2) -- (p3) -- (p4) -- cycle;
  \path[fill=orange, opacity=0.3, rounded corners=20pt ] (-1,3) -- (3.2,3.5) --(4.5,3.8) -- (3.8,2.5) -- (3.5,-0.4) -- (2.7,1.3) -- (1, 1) -- cycle;
  \node (B) at (0,2.8) {$\beta$};
  
  \draw [] (7,1)--(7,3);
  \draw[] (6,2) --(8,2);
  \node(o) at (7,2){};
  \draw[thick, color=magenta] ($(o)+(0,0.7)$) arc[radius=0.7, start angle=90, end angle=0];
  \draw[thick, color=black] ($(o)+(0,0.7)$) arc[radius=0.7, start angle=90, end angle=360];
 \node[color=magenta](w) at (8.2,2.5) {$\omega$};
   \node[color=black](s) at (6,2.8) {$S^{1}$};
  
   \draw[fill=yellow, color=yellow,opacity=0.7] (p2) -- ($(p2)+(0,0.4)$) -- ($(p3)+(0,0.4)$) -- (p3) -- (p2) ;
  \filldraw[thick, color=yellow] ($(p3)+(0,0.4)$) arc[radius=0.4, start angle=90, end angle=0] -- (p9) -- (p3);
 
  \node (P) at (2,1) {$P$};
  \draw [] (7,1)--(7,3);
  \draw[] (6,2) --(8,2);
  \node(o) at (6,2){};
  \draw[fill=yellow, color=yellow,opacity=0.6] (p3) -- ($(p3)+(0.4,0)$) -- ($(p4)+(0.4,0.62)$) -- ($(p4)+(0,0.62)$) -- (p3) ;
	\draw[->, thick,blue] ($(p3)+(0.5,0.9)$) -- ($(p3)+(0.5,1.4)$);
	\draw[->,thick, blue] ($(p3)+(0.5,0.9)$) -- ($(p3)+(1,0.9)$);
    \draw[blue, pattern color=blue, pattern = north west lines, opacity=0.4] ($(p3)+(0.5,0.9)$) -- ($(p3)+(0.5,1.35)$) -- ($(p3)+(0.95,0.9)$) -- cycle;
    \draw[->, thick,blue] ($(p3)+(0,-1)$) -- ($(p3)+(0.4,-1)$);
     \draw[->, thick,blue] (1.5, 2) -- (1.5,2.4);
      \draw[dotted,thick,magenta] (p3) edge node[left]{\color{magenta}\footnotesize $\varepsilon$} ($(p3)+(0,0.42)$);
     \draw[dotted, thick, magenta] (p3) edge node[below]{\color{magenta}\footnotesize $\varepsilon$} ($(p3)+(0.42,0)$);
     \node [color=blue] at ( $(p3)+(1,0.5)$) {\footnotesize $N(v_{23},P)$};
      \node[above, color=blue]  at (1.7, 2.3) {\footnotesize $N(F_{2},P)$};
     \node[right, color=blue]  at ($(p3)+(0.4,-1)$) {\footnotesize $N(F_3,P)$};
     
\end{tikzpicture}
\end{center}
where we can see that the volume of the local parallel set is $\varepsilon$ times the $1 = d-1$ dimensional volume of the facets (1-dimensional faces) of $P$ which intersect $\beta$ and whose normals intersect $\omega$. This gives us the two rectangles, then we add the rounded corner, whose volume will be $\varepsilon^2$ times a quantity involving the $0$-dimensional volume of vertex ($0$-dimensional face) $v$ and the $1$-dimensional measure of the normal cone of $v$ intersected with $\omega\in S^1$. 
\end{example}
}

So for {a $d$-dimensional polytope $P$}, we get Minkowski tensors of the form 
\begin{equation}
\Phi_j^{r,s}(P) = \frac{1}{r!s!\,\omega_{d-j+s}}\int_{\RR^d\times S^{d-1}}
\raisebox{1pt}{ \hspace{-17pt}$x^ru^s$} \sum_{F\in\FF_j(P)}\int_{F}\int_{N(F,P)\cap S^{d-1}} \raisebox{1pt}{ \hspace{-35pt} $\mathbf{1}_{d(x,u)}(x,u)\HH^{d-j-1}(du)\HH^j(dx)$}
\label{EQ:polytope_tensors} \end{equation}
for $j=0,\ldots,d-1$ and
\begin{equation}
\Phi_d^{r,0}(P) = \frac{1}{r!}\int_P x^r\lambda(dx).
\label{EQ:poly_vol_tensor} \end{equation}

\begin{example}\label{EX:quad}
Let $P$ be a quadrilateral with vertices $(0,0), (a,0), (a,b)$, and $(0,b)$. Then its volume tensors are of the form
\begin{align*}
\Phi_2^{r,0}(P) & =\frac{1}{r!}\int_{0}^a\int_0^b (x_1e_1+x_2e_2)^r dx_2dx_1 \\
		& = \frac{1}{r!}\int_0^a\int_0^b \sum_{k=0}^r \binom{r}{k} x_1^kx_2^{r-k}(e_1^k\otimes e_2^{r-k}) dx_2dx_1 \\
		& = \frac{1}{r!}\sum_{k=0}^r\binom{r}{k} \int_0^a \frac{1}{r-k+1}x_1^kb^{r-k+1}(e_1^k\otimes e_2^{r-k}) dx_1 \\
		& = \frac{1}{r!}\sum_{k=0}^r\binom{r}{k} \frac{1}{(k+1)(r-k+1)} a^{k+1}b^{r-k+1}(e_1^k\otimes e_2^{r-k}).
\end{align*}
That is $\Phi_2^{0,0} = V_2(P) = ab$, $\Phi_2^{1,0} = \frac{1}{2}ab(a,b)^\top$, and {$\Phi_2^{2,0} = \frac{1}{24}ab\begin{pmatrix}
4a^2 & 3ab \\ 3ab & 4b^2
\end{pmatrix}$.} 
And letting $u_F$ be the outer normal to facet $F$ of $P$, we have that the surface tensors are 
\begin{align*}
\Phi_1^{0,s}(P) & = \frac{1}{s!\,\omega_{1+s}}\int_{\RR^2\times S^{1}}
\raisebox{1pt}{ \hspace{-17pt}$u^s$} \sum_{F\in\FF_1(P)}\int_{F}\int_{N(F,P)\cap S^{1}} \raisebox{1pt}{ \hspace{-30pt} $\mathbf{1}_{d(x,u)}(x,u)\HH^{0}(du)\HH^1(dx)$} \\
		& = \frac{1}{s!\,\omega_{1+s}}\int_{\RR^2\times S^{1}}
\raisebox{1pt}{ \hspace{-17pt}$u^s$} \sum_{F\in\FF_1(P)}\int_{F}{\mathbf{1}_{d(x,u)}(x)}\delta_{u_F}(u)\HH^1(dx) \\
		& = \frac{1}{s!\,\omega_{1+s}} \int_{S^1}u^s \sum_{F\in \FF_1(P)} V_1(F)\delta_{u_F}(u) \\
		& = \frac{1}{s!\,\omega_{1+s}} \sum_{F\in \FF_1(P)} V_1(F)u_F^s.
\end{align*}
That is, 
\begin{align*}
\Phi_1^{0,0}(P) & = \frac{1}{\omega_1}(a+b+a+b) \hspace{64pt} = a+b \\
\Phi_1^{0,1}(P) & = \frac{1}{\omega_2}\left(a(-e_2)+be_1+ae_2+b(-e_1)\right) = 0\\
\Phi_1^{0,2}(P) & = \frac{1}{2!\,\omega_3}\left(a(-e_2\otimes -e_2)+b(e_1\otimes e_1)+a(e_2\otimes e_2)+b(-e_1\otimes -e_1)\right) \\
	& = \frac{1}{8\pi}\left(\begin{array}{cc} a & 0 \\ 0 & b \end{array}\right).
\end{align*}
{We note that the computation of $\Phi^{2,0}_1(P)$ using the software Sage is given in Appendix~\ref{sec:sage}.}
\end{example}

Following calculations similar to the above example, we consider the surface tensors of~$P$, and show how they relate to the tensors of the facets of $P$. Let $u_F$ denote the outer unit normal to facet $F$ of $P$. Then 
\begin{align}
\Phi_{d-1}^{r,s}(P) & = \frac{1}{r!s!\,\omega_{1+s}}\int_{\RR^d\times S^{d-1}}
\raisebox{1pt}{ \hspace{-15pt}$x^ru^s$} \sum_{F\in\FF_{d-1}(P)}\int_{F}\int_{N(F,P)\cap S^{d-1}} \raisebox{1pt}{ \hspace{-35pt} $\mathbf{1}_{d(x,u)}(x,u)\HH^{0}(du)\HH^{d-1}(dx)$} \nonumber\\
		& = \frac{1}{r!s!\,\omega_{1+s}}\int_{\RR^d\times S^{d-1}}
\raisebox{1pt}{ \hspace{-15pt}$x^ru^s$} \sum_{F\in\FF_{d-1}(P)}\int_{F} \raisebox{1pt}{ $\mathbf{1}_{d(x,u)}\delta_{u_F}(u) \HH^{d-1}(dx)$}\nonumber \\
		& = \frac{1}{r!s!\,\omega_{1+s}}\int_{\RR^d} x^r\sum_{F\in\FF_{d-1}(P)}\left( \int_F{\mathbf{1}_{d(x,u)}}(x,u_F)\HH^{d-1}(dx) \right)\otimes u_F^s \nonumber\\
		& = \frac{1}{s!\,\omega_{1+s}}\sum_{F\in\FF_{d-1}(P)}\left( \frac{1}{r!}\int_Fx^r\HH^{d-1}(dx) \right)\otimes u_F^s \nonumber\\
		& = \frac{1}{s!\,\omega_{1+s}}\sum_{F\in\FF_{d-1}(P)} \Phi_{d-1}^{r,0}(F)u_F^s, \label{EQ:poly_surface_tensors}
\end{align}
where we use the fact that the $d-1$ support measure is only nonzero on the part of the normal cycle of the $P$ of the form $(x,u_F)$ where $x\in F$ for facets $F$ of $P$. Hence the integration over $\RR^d$ is really an integration over each facet which gives us that facet's ($(d-1)$-dimensional) volume tensors. 

\section{Adjoint polynomial} \label{sec:moments}

We have already seen that certain volume tensors are related to frequently studied moments of a convex body such as its center of mass and moment of inertia. {As we will explain below, one can see} that the volume tensors are in general rescalings of the moments of a convex body. 

For an $d$-dimensional polytope $P \subset \RR^d$, let $\mu_P$ denote the uniform probability measure on $P$. 
\begin{defn} The {\em moments} $m_I$ of the distribution $\mu_P$ are the expected values of the monomials $x^I$. That is, 
\begin{equation} m_{i_1\ldots i_d} = \int_{\RR^d} x_1^{i_1}\cdots x_d^{i_d} d\mu_P, \quad i_1,\ldots, i_d\in\ZZ_{\geq 0}. \label{EQ:moments} \end{equation}
\end{defn}

Notice that the uniform probability distribution on $P$ has probabilty density function $f_P(x) = \frac{1}{V_d(P)}$ for $x\in P$ and $f_P(x)=0$ elsewhere, so that 
\begin{equation}
m_{i_1\ldots i_d} = \frac{1}{V_d(P)}\int_P x_1^{i_1}\cdots x_d^{i_d}\, \lambda(dx).
\label{EQ:moments2}\end{equation}

{Using basic multinomial expansion to write the rank $r$ tensor $x^r$ for $x\in\RR^d$ in terms of tensors in the standard basis $\{e_i\}_{i=1}^d$, we see each coefficient is some monomial that we would integrate to get an $r$th moment, namely
$$x^r = (x_1e_1+\cdots x_de_d)^r = \sum_{i_1+\cdots+i_d = r} \binom{r}{i_1,\ldots,i_d} x_1^{i_1}\cdots x_d^{i_d}(e_1^{i_1}\otimes\cdots\otimes e_d^{i_d}).$$
Notice that the multinomial coefficient counts the number of components of the tensor in array notation that will have the same value by symmetry.}
 
Now fix an integer $r>0$ and let $\{i_1,\ldots,i_d\}$ range over the set of partitions of $r$ of size $d$; that is, $i_1+\cdots+i_d = r$ and $i_1,\ldots, i_d\in\ZZ_{\geq 0}$. Then using \eqref{EQ:moments2} we can reinterpret the Minkowski volume tensor of $P$ as recording the collection of $r$th moments of $\mu_P$.
\begin{prop} For a $d$-dimensional polytope $P$, the rank $r$ Minkowski volume tensor and the $r$th moments of the uniform distribution are related as follows. 
\begin{align}
\Phi_d^{r,0}(P) & = \frac{1}{r!}\int_P x^r\lambda(dx) \nonumber\\
		& = \frac{1}{r!} \int_P \sum_{i_1+\cdots+i_d=r} \frac{r!}{i_1!\cdots i_d!}(x_1^{i_1}\cdots x_d^{i_d})(e_1^{i_1}\otimes\cdots\otimes e_d^{i_d}) \lambda(dx) \nonumber\\ 
		& = V_d(P)  \sum_{i_1+\cdots+i_d=r} \frac{1}{i_1!\cdots i_d!}m_{i_1\ldots i_d}(e_1^{i_1}\otimes\cdots\otimes e_d^{i_d}). \label{EQ:tensor_as_moments}
 \end{align}
 \label{PROP:tensor_as_moments} \end{prop}
 
\begin{example}
 Consider the polytope $P=\conv\left\{(0,0),(2,2),(3,2), (5,0) \right\}$. 
\begin{center}
    \begin{tikzpicture}
      \draw[color=orange, fill=orange, opacity=0.6](0,0) -- (2,2) -- (3,2) -- (5,0) -- cycle;
      \node () at (2.5,1){$P$};
    \end{tikzpicture}
\end{center}
\begin{align*}
\Phi_{2}^{1,0} & = \int_{P}x\,\lambda(dx) \\ 
	& =\int_{P}\left(\begin{array}{@{\,}c@{\;\;}c@{\,}}x_{1}&x_{2}\end{array}\right)^\top dx_{2}dx_{1} \\
	& =\int_{0}^{2}\!\!\int_{0}^{x_{1}}\left(\begin{array}{@{\,}c@{\;\;}c@{\,}}x_{1}&x_{2}\end{array}\right)^\top dx_{2}dx_{1}
	+\int_{3}^{5}\!\!\int_{0}^{-x_{1}+5}\hspace{-10pt}\left(\begin{array}{@{\,}c@{\;\;}c@{\,}}x_{1}&x_{2}\end{array}\right)^\top dx_{2}dx_{1} 
	+\int_{2}^{3}\!\!\int_{0}^{2}\left(\begin{array}{@{\,}c@{\;\;}c@{\,}}x_{1}&x_{2}\end{array}\right)^\top dx_{2}dx_{1} \\
	& =\left(\begin{array}{@{\,}c@{\;\;}c@{\,}} 15& \frac{14}{3} \end{array} \right) \\
	& =6\left(\begin{array}{@{\,}c@{\;\;}c@{\,}}m_{10}& m_{01}\end{array}\right)
\end{align*}
Considering the above calculations componentwise, it is clear that we are calculating the first moments of the distribution of $\mu_{P}$. From these moments we get the center of mass, which is $(m_{10},m_{01}) = (\frac{5}{2},\frac{7}{9})$. These moments are simply the components of our tensors normalized by $V(P)=6$.
\end{example}
 
\begin{example} Recall for $P=\conv\{(0,0),(a,0),(a,b),(0,b)\}$, we have $\Phi_2^{1,0} = \frac{1}{2}ab(a,b)^\top$. So that $m_{10} = \frac{1}{2}a, m_{01}=\frac{1}{2}b$, since $V_2(P)=ab$. This is, of course, also consistent with the fact that the first moment should represent the expected value or center of mass. 
\end{example}

In \cite{momentpaper} the following closed form rational generating function for the moments of a simplicial polytope $P$ is given. 

\begin{theorem}[Theorem 2.2 \cite{momentpaper}] The normalized moment generating function for the uniform probability distribution $\mu_P$ on simplicial polytope $P = \conv\{x_1,\ldots,x_m\}$ {with each $x_k = (x_{k,1},\ldots,x_{k,d})\in\RR^d$} is 
\begin{multline}
\sum_{i_1,\ldots,i_d\in\mathbb{N}}\frac{(i_1+\cdots+i_d+d)!}{i_1!\cdots i_d!d!}m_{i_1\ldots i_d} t_1^{i_1}\cdots t_d^{i_d} \\ 
= \frac{1}{V_d(P)}\sum_{\sigma\in\Sigma} \frac{V_d(\sigma)}{\prod_{k\in\sigma}(1-x_{k,1}t_1-\cdots - x_{k,d}t_d)}, 
\label{EQ:mom_gen_fun}\end{multline}
where $\Sigma$ is a triangulation of $P$. Moreover this generating function is independent of the triangulation $\Sigma$. 
\end{theorem}

Writing the rational generating function of \eqref{EQ:mom_gen_fun} over a common denominator gives as a numerator an inhomogeneous polynomial of degree at most $m-d-1$ in the variables $t_1,\ldots, t_d$ whose coefficients depend only on the coordinates of the vertices $x_1,\ldots, x_m$ and not the triangulation $\Sigma$. This polynomial is called the {\em adjoint} of $P$ and is denoted $\Ad_P$.  
\begin{equation}
\Ad_P(t_1,\ldots,t_d) := \sum_{\sigma\in\Sigma}\frac{V_d(\sigma)}{V_d(P)}\prod_{k\notin\sigma} (1-x_{k,1}t_1-\cdots - x_{k,d}t_d)
\label{EQ:adjoint}\end{equation}
The adjoint polynomial of a polytope was introduced by Warren \cite{Warren_1996} in the setting of geometric modeling to study barycentric coordinates. 

\begin{defn} The {\em non-face subspace arrangement} of polytope $P = \conv\{x_1,\ldots,x_m\}\subset\RR^d$ is the collection of affine linear spaces in $\RR^d$ of the form 
\begin{equation}
L_\tau = \left\{(t_1,\ldots, t_d): \sum_{j=1}^dx_{k,j}t_j = 1, \mbox{ for all }k\in\tau\right\}
\label{EQ:nonface_space}\end{equation}
where $\tau$ ranges over all subsets of $\{1,\ldots,m\}$ such that $\{x_k:k\in \tau\}$ is not the vertex set of a face of $P$. This collection {of all these $L_\tau$}
is denoted $\mathcal{NF}(P)$.
\end{defn}

\begin{corollary}[Corollary 2.5 \cite{momentpaper}]
The adjoint $\Ad_P$ is a polynomial of degree at most $m-d-1$ that vanishes on $\mathcal{NF}(P)$.
\label{COR:adjoint_vanish}\end{corollary}

\begin{example} Consider the square $P = \conv\{(2,1),(2,3),(4,3),(4,1)\}$. Its nonfaces are the sets $\{x_1,x_3\}$, $\{x_2,x_4\}$, and any set of $3$ vertices. Then for nonface indexed by $\tau=\{2,4\}$, we have $L_\tau$ is the intersection, $(1/5,1/5)$, of the two lines $\LL_{v_2}$: $2t_1+3t_2=1$, and $\LL_{v_4}$: $4t_1+t_2=1$, pictured below left. Similarly one can check that $L_{\{1,3\}} = (1,-1)$. Any $\tau$ of size $3$ has $L_\tau=\emptyset$. 

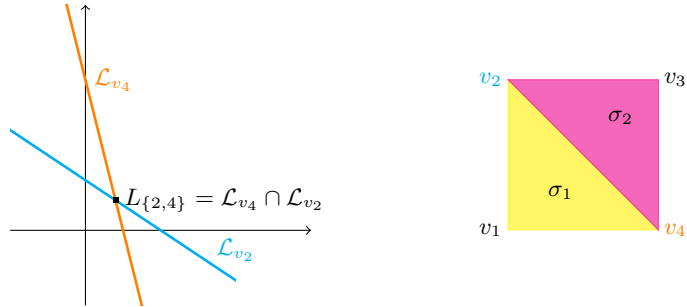
 \begin{center}
     \begin{tikzpicture}[x=2cm, y=2cm]
          \node[right] at (0.2,0.2){$L_{\{2,4\}} = \LL_{v_{4}}\cap \LL_{v_{2}}$};
        \draw[->] (-1/2,0) -- (1.5,0);
        \draw[->] (0,-1/2) -- (0,1.5);
    \node[right, color=orange] at (0,1) {$\LL_{v_{4}}$};
    \node[below, color=cyan] at (1,0) {$\LL_{v_{2}}$};
      \clip(-1/2.,-1/2.) rectangle (1,1.5);
    \draw [line width=1pt, color=cyan] plot(\x,{(--1-2*\x)/3});
    \draw [line width=1pt, color=orange] plot(\x,{(--1-4*\x)/1});
    \node[fill=black,inner sep=1pt] at (0.2,0.2) {};
      \end{tikzpicture}\hspace{50pt}
             \begin{tikzpicture}
         \coordinate[] (a) at (0,0);
         \coordinate [](b) at  (0,2);
         \coordinate [] (c) at (2,2);
         \coordinate[] (d) at (2,0);
                  
         \node[ left, shape=circle, inner sep=1pt, color=black] (na) at (0,0) {$v_{1}$};
         \node [ left, shape=circle, inner sep=1pt, color=cyan ](nb) at  (0,2) {$v_{2}$};
         \node [right, shape=circle, inner sep=1pt, color=black] (nc) at (2,2){$v_{3}$};
         \node[right, shape=circle, inner sep=1pt, color=orange] (nd) at (2,0) {$v_{4}$};
         
         \draw [fill=pink, color=yellow, opacity=0.6] (a) -- (b) -- (d) -- cycle;
         \draw [fill=pink, color=magenta, opacity=0.6] (b) -- (c) -- (d) -- cycle;
         \node [] at ($(a)+(0.7,0.5)$) {$\sigma_{1}$};
         \node[] at ($(d)+(-0.5,1.5)$) {$\sigma_2$};      
         
         \draw (0,-0.9) node[] {}; 
       \end{tikzpicture}
    \end{center}

Now using the triangulation pictured above right, we calculate the adjoint of $P$ as 
$$\frac{1}{2}(1-2t_1-t_2) + \frac{1}{2}(1-4t_1-3t_2) = 1-3t_1-2t_2.$$ This is degree $4-2-1$ and vanishes on $L_{\{1,3\}}$ and $L_{\{2,4\}}$.
\end{example}

In \cite{momentpaper} it was shown that not only does the adjoint vanish on the nonface arrangement of $P$, but it is the unique polynomial (up to scalar multiples) of degree $m-d-1$ that vanishes there. 

\begin{theorem}[Theorem 2.6 \cite{momentpaper}]
If the projective closure $\mathcal{H}_{P^*}\subset\mathbb{P}^d$ of the hyperplane arrangement formed by the linear spans of the facets of the dual polytope $P^*$ is simple, then there is a unique hypersurface of degree $m-d-1$ which vanishes along the projective closure of $\mathcal{NF}(P)$. The defining polynomial of this hypersurface is the adjoint of $P$.
\label{THM:adjoint_unique}
\end{theorem}

It is noted in \cite{momentpaper} that the map $P\mapsto Ad_P$ represents the computation of all moments of $P$ and induces a polynomial map $X\mapsto Ad_X$ on an open dense set of matrices $X\in\RR^{m\times d}$. {They then define the {\em adjoint moment variety} $\mathcal{M}_{Ad}(P)$ to be the Zariski closure of the image of this map in complex projective space $\mathbb{P}^{\binom{m-1}{d}-1}$.} This variety can be viewed as a moduli space of Wachspress varieties (for more on these varieties see \cite{Wachspress}).

\section{Surface Adjoint Polynomial} \label{sec:genfn}

We have just seen how Minkowski volume tensors are connected to moments of a probability distribution, and that an efficient way of recording this connection is via the adjoint polynomial. In this section we take the next natural step and consider Minkowski surface tensors. We will show how the generating function of Section 3 can be used to build a generating function for Minkowski surface tensors, and that a natural analog of the adjoint polynomial results. 

From \eqref{EQ:tensor_as_moments} we see that we can think of the generating function in \eqref{EQ:mom_gen_fun} as recording rank~$r$ volume tensors as its graded (by total degree) pieces. 
\begin{prop}
The generating function for moments of a polytope $P$ given in Theorem~\ref{THM:genfun} has as its degree $r$ graded piece the rank $r$ Minkowski volume tensor of $P$. 
\label{PROP:graded_gen_fun}\end{prop}

\begin{proof} Denote by $\Phi_j^{r,s}(P)(\bt)$ the polynomial in $\RR[t_1,\ldots,t_d]$ obtained by {making the substitution $e_i\mapsto t_i$ and replacing the tensor product with regular multiplication.}
Then from \eqref{EQ:tensor_as_moments} and \eqref{EQ:mom_gen_fun} we immediately get\begin{multline}
\sum_{\sigma\in\Sigma} \frac{V_d(\sigma)}{\prod_{k\in\sigma}(1-x_{k,1}t_1-\cdots - x_{k,d}t_d)} \\
=\sum_{r\geq 0}\left(V_d(P) \sum_{\stackrel{i_1,\ldots, i_d\in\mathbb{N}}{i_1+\cdots+i_d=r}}\frac{(r+d)!}{i_1!\cdots i_d!d!}m_{i_1\ldots i_d} t_1^{i_1}\cdots t_d^{i_d}\right) \\
=\sum_{r\geq 0}\left(\frac{(r+d)!}{d!} \Phi_d^{r,0}(P)(\bt) \right).
\label{EQ:graded_gen_fun}\end{multline}

\end{proof}
\begin{remark}  Notice that by symmetry of the tensors, several components of the associated array will be collected into a single term of the tensor polynomial.
\end{remark}

\begin{example} {\color{black}
Recall from Example~\ref{EX:quad} that
\begin{equation*}
\Phi_2^{r,0}(P) 
		= \frac{1}{r!}\sum_{k=0}^r\binom{r}{k} \frac{1}{(k+1)(r-k+1)} a^{k+1}b^{r-k+1}(e_1^k\otimes e_2^{r-k}).
\end{equation*}
Then we get that the tensor polynomial is
\begin{equation*}
\Phi_2^{r,0}(P)(\bt)
		= \frac{1}{r!}\sum_{k=0}^r\binom{r}{k} \frac{1}{(k+1)(r-k+1)} a^{k+1}b^{r-k+1}\cdot t_1^k t_2^{r-k}.
\end{equation*}
}
\end{example}

Let us now look at what happens when we consider surface tensors. By slight abuse of terminology, {we will call tensors of the form $\Phi^{r,s}_{d-1}$ surface tensors for all values of $r\geq 0$.}
Now recall from~\eqref{EQ:poly_surface_tensors} we have for each $r\geq 0$,
$$\Phi_{d-1}^{r,s}(P) = \frac{1}{s!\,\omega_{1+s}}\sum_{F\in\FF_{d-1}(P)} \Phi_{d-1}^{r,0}(F)u_F^s.$$
So now fixing an $s\geq 0$, we can consider the generating function for all rank $r+s$ surface tensors of~$P$ using \eqref{EQ:graded_gen_fun}
\begin{align*}
\sum_{r\geq 0} \frac{(r+d-1)!}{(d-1)!}\Phi_{d-1}^{r,s}(P)(\bt) 
	& = \sum_{r\geq 0} \frac{1}{s!\,\omega_{1+s}}\sum_{F\in\FF_{d-1}(P)} \frac{(r+d-1)!}{(d-1)!}\Phi_{d-1}^{r,0}(F)(\bt)u_F^s \\
	& = \frac{1}{s!\,\omega_{1+s}}\sum_{F\in\FF_{d-1}(P)} \sum_{\sigma\in\Sigma_F}\frac{V_{d-1}(\sigma)u_F^s}{\prod_{k\in\sigma}(1-x_{k,1}t_1-\cdots-x_{k,d}t_d)},
\end{align*}
which we rewrite using the adjoint of each facet to get the following result. 

\begin{theorem} Fix an integer $s\geq 0$. Let $P = \conv\{x_1,\ldots,x_m\}\subset\RR^d$ be a polytope with simplicial facets. Let $\LL_k := 1-x_{k,1}t_1-\cdots-x_{k,d}t_d$ {for $k=1,\ldots,m$, then the generating function for rank $r+s$ Minkowski surface tensors of $P$ as $r$ ranges over $\ZZ_{\geq 0}$} is denoted $MG_P(\bt)$ and given as follows
\begin{align*}
\sum_{r\geq 0} \frac{(r+d-1)!}{(d-1)!}\Phi_{d-1}^{r,s}(P)(\bt) 
	& = \frac{1}{s!\,\omega_{1+s}}\sum_{F\in\FF_{d-1}(P)} \frac{V_{d-1}(F)\Ad_F(t_1,\ldots,t_d)u_F^s}{\prod_{k\in F}\LL_k} \\
	& = \frac{1}{s!\,\omega_{1+s}}\frac{\sum_{F\in\FF_{d-1}(P)} V_{d-1}(F)\Ad_F(t_1,\ldots,t_d)u_F^s\prod_{k\notin F}\LL_k}{\prod_{k=1}^m\LL_k}.
\end{align*}
\label{THM:genfun}
\end{theorem}

We note that whereas the generating function of \cite{momentpaper} is only stated to be valid for simplicial polytopes, here we can define the surface adjoint for any polytope whose facets are simplicial polytopes.

\begin{defn} Call the numerator of the above generating function the {\em surface adjoint of $P$}, and denote it by $\alpha_P^s$. If $P\subset\RR^d$ has $m$ vertices $x_1,\ldots,x_m$ and $f$ facets $F_1,\ldots, F_f$ (which are themselves simplicial polytopes), then 
$$\alpha_P^s(t_1,\ldots,t_d,u_1,\ldots,u_f) := \sum_{i=1}^f V_{d-1}(F_i)\Ad_{F_i}(t_1,\ldots, t_d)u_i^s\prod_{k\notin F_i}\LL_k.$$
\end{defn}

\begin{remark} We must be careful to note that in the case of surface adjoints there is a subtle distinction between the adjoint as a polynomial and the numerator of the generating function. In the generating function, $u_F^s$ is the $s$-fold tensor of the unit normal to facet $F$ of $P$. In the surface adjoint $\alpha_P^s$ we regard $u_F$ as a variable of the polynomial ring $\RR[u_1,\ldots,u_f][t_1,\ldots,t_d]$. In what follows we will mostly be concerned with $\alpha_P^s$ as a polynomial in $\bt$ {since it is the coefficient of the monomials in $\bt$ that give us tensor components}. Thus we will endow the above ring with the grading given by $\deg(t_i)=1$ for $i=1,\ldots,d$, and $\deg(u_F) = 0$ for $F=1,\ldots,f$.
\end{remark}

\begin{example} Let $P = \conv\{(1,2),(-1,1),(-2,-1),(1,-1)\}$. Since each facet of $P$ is a $1$-dimensional simplex, $\Ad_{F_i} = 1$ for each $i=1,2,3,4$. Then 
\begin{multline*} \alpha_P^s = \sqrt{5}(1+2t_1+t_2)(1-t_1+t_2)u_1^s + \sqrt{5}(1-t_1-2t_2)(1-t_1+t_2)u_2^s \\
+ 3(1-t_1-2t_2)(1+t_1-t_2)u_3^s+3(1+2t_1+t_2)(1+t_1-t_2)u_4^s.\end{multline*}
\end{example}

\begin{prop}
The surface adjoint $\alpha_P^s$ is a polynomial of degree at most $m-d$ that vanishes on the union of $\mathcal{NF}(F)$ for facets F of P.
\label{PROP:surface_adjoint_vanish}
\end{prop}

\begin{proof} Fix some facet $F$ of $P$. We will show that each {summand (corresponding to some facet of $P$)} of the surface adjoint $\alpha_P^s$ vanishes on the nonface subspace arrangement of $F$. First $\Ad_F(t_1,\ldots,t_d)$ vanishes on $\mathcal{NF}(F)$ by Corollary~\ref{COR:adjoint_vanish}. Now for each facet $F'\neq F$ of $P$, every nonface $\tau$ of $F$ satisfies one of 
\begin{enumerate}[(i).]
\item $\tau$ is a nonface of $F'$
\item $\tau$ contains some vertex $x_\ell\in F\setminus F'$.
\end{enumerate}
If we are in case (i), then $L_\tau\subset\mathcal{NF}(F')$, so that $\Ad_{F'}$ vanishes on $L_\tau$ by Corollary~\ref{COR:adjoint_vanish}. If we are in case (ii), then $\LL_\ell = 1-x_{\ell,1}t_1-\cdots-x_{\ell,d}t_d$ vanishes on $L_\tau$, since $L_\tau = \{(t_1,\ldots,t_d) : \LL_k(t) = 0,\forall k\in\tau\}$. Then the summand $V_{d-1}(F')\Ad_{F'}(t_1,\ldots,t_d)\prod_{k\notin F'}\LL_k$ vanishes on $L_\tau$, since $\LL_\ell$ with $\ell\notin F'$ does. Since the $F'$ summand thus vanishes on $L_\tau$ for all nonfaces $\tau$ of~$F$, it vanishes on $\mathcal{NF}(F)$, as required. 
\end{proof}

\begin{example} Consider a perturbed cube $P$ with vertices $(1,1,2)$, \\ 
\begin{minipage}{0.52\textwidth} $(1,-1,1)$, $(1,-2,-1)$, $(1,1,-1)$, $(\frac{1}{2},1,2)$, $(-1,-1,1)$, $(-2,-2,-1)$, $(-2,1,-1)$. 
Then, for example, we have $\LL_1 = 1-t_1-t_2-2t_3$, and one can compute that the surface adjoint of $P$ is 
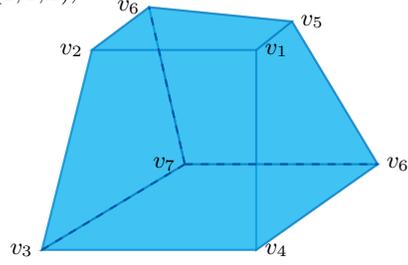
\begin{align*}
\alpha_P^s & = 6u_{F_1}^s\left(1-t_1+\frac{1}{2}t_2-\frac{1}{2}t_3\right)\LL_5\LL_6\LL_7\LL_8 \\
	&\phantom{=}+ \frac{7}{4}u_{F_2}^s\left(1-\frac{1}{7}t_1-\frac{1}{7}t_2-\frac{11}{7}t_3\right)\LL_3\LL_4\LL_7\LL_8 \\
	&\phantom{=} +\frac{5}{2}u_{F_3}^s\left(1+\frac{1}{5}t_1+\frac{7}{5}t_2-\frac{1}{5}t_3\right)\LL_1\LL_4\LL_5\LL_8 +9u_{F_4}^s\left(1+\frac{1}{2}t_1+\frac{1}{2}t_2+t_3\right)\LL_1\LL_2\LL_5\LL_6 \\
	&\phantom{=}+ \frac{27}{4}u_{F_5}^s\left(1-t_2-t_3\right)\LL_2\LL_3\LL_6\LL_7 + 3u_{F_6}^s\left(1+\frac{5}{4}t_1+\frac{1}{2}t_2-\frac{1}{2}t_3\right)\LL_1\LL_2\LL_3\LL_4. 
\end{align*}
\end{minipage}\begin{minipage}[t][10pt][b]{0.4\textwidth}
\hspace{-120pt}\vspace{-10pt}
       \begin{tikzpicture}[scale=0.95]
    \coordinate(123) at (-0.3,-2); %(0,-2);
    \coordinate (124) at (0.4,0.8); %(0.4,0.8);
    \coordinate (134) at (2.7,0.8);%(2.7,0.8);
    \coordinate (133) at (2.7,-2);%(3,-2);
    \coordinate (223) at (1.7,-0.8);%(1.2,-1);
    \coordinate (233) at (4.4,-0.8);%(4.2,-1);
    \coordinate (234) at (3.2,1.2);%(3.2,1.2);
    \coordinate (224) at (1.2,1.4);%(1.2,1.4);
   
        \draw[color=cyan, fill=cyan, thick, opacity=0.5] (123)-- (124) -- (134) -- (133) -- cycle; % back
    \draw [color=cyan!50!blue, thick, fill=cyan, opacity=0.5] (223)-- (233) -- (234) -- (224) -- cycle; % front
    \draw [color=cyan!50!blue, thick, fill=cyan, opacity=0.5] (124)-- (224) -- (223) -- (123) -- cycle; %left
    \draw [color=cyan!50!blue, thick, fill=cyan, opacity=0.5] (134)-- (234) -- (233) -- (133) -- cycle;  %right
    \draw [color=cyan!50!blue, thick, fill=cyan, opacity=0.5] (124)-- (134) -- (234) -- (224) -- cycle; % up
    \draw [color=cyan!50!blue, thick, fill=cyan, opacity=0.5] (123)-- (223) -- (233) -- (133) -- cycle;%down
    \draw[dashed, thick, color=cyan!30!blue] (123) -- (223);
    \draw[dashed, thick, color=cyan!30!blue] (223) -- (233);
    \draw[dashed, thick, color=cyan!30!blue] (223) -- (224);
    \node[left] at (123) {$v_{3}$};
    \node[left] at (124) {$v_{2}$};
    \node[right] at (134) {$v_{1}$};
    \node[right] at (133) {$v_{4}$};
    \node[left] at (223) {$v_{7}$};
    \node[right] at (233) {$v_{6}$};
    \node[right] at (234) {$v_{5}$};
    \node[left] at (224) {$v_{6}$};
    \end{tikzpicture}
\end{minipage}

Now looking at the nonface subspace arrangement of $F_1$, we see that the nonfaces are indexed by $\tau_1=\{1,3\}$ and $\tau_2=\{2,4\}$ (nonface subspaces indexed by larger sets, say $\{1,2,3\}$, would of course be contained in those indexed by smaller ones, say $\tau_1$, so we don't include them separately). Thus we have 
$$L_{\tau_1} = \{(t_1,t_2,t_3) : \LL_1 = 0, \LL_3=0\} = \{(t_1,t_2,t_3): t_2=-t_3, t_1=-t_3+1\}$$
$$L_{\tau_2} = \{(t_1,t_2,t_3) : \LL_2 = 0, \LL_4=0\} = \{(t_1,t_2,t_3): t_2=t_3, t_1=1\}.$$
It is easy to check that the adjoint of $F_1$ vanishes on these subspaces, as claimed, so that the first term of $\alpha_P^s$ vanishes on $L_{\tau_1}$ and $L_{\tau_2}$. Now the remainder of the terms contain at least one of $\LL_1$ or $\LL_3$ and at least one of $\LL_2$ and $\LL_4$, each of which vanish on $L_{\tau_1}$ and $L_{\tau_2}$, respectively, by definition. Thus, the whole $\alpha_P^s$ vanishes on the nonface subspace arrangement $\mathcal{NF}(F_1)=L_{\tau_1}\cup L_{\tau_2}$, as expected.
\label{EX:cube}\end{example}

Notice that Proposition~\ref{PROP:surface_adjoint_vanish} is the analog of Corollary~\ref{COR:adjoint_vanish} for surface adjoints. It is then natural to ask if there is an analog of Theorem~\ref{THM:adjoint_unique} for surface adjoints.

\section{Minkowski tensors from their generating function} \label{sec:simplicial}

The strength of connecting Minkowski tensors to moments via a generating function is that we are now able to use generating function techniques to manipulate and access these (generally quite complicated) tensors. In this section we focus  on one technique for extracting coefficients from a generating function and show how it can be used to give an explicit expression for the Minkowski surface tensors of simplicial polytopes. 

In general, if we have a closed form multivariate generating function 
$$F(t_1,\ldots,t_d) = \sum_{i_1,\ldots,i_d\geq 0} f_{i_1\ldots i_d}t_1^{i_1}\cdots t_d^{i_d},$$ 
we can calculate the coefficient $f_{i_1\ldots i_d}$ by taking the partial derivative $\d^{i_1}/\d t_{1}^{i_1}\cdots\d^{i_d}/\d t_{d}^{i_d}$ divided by $i_1!\cdots i_d!$ and then evaluating at $(0,\cdots,0)$. 

\begin{example} If $F(x,y) = \frac{1}{1-x-y}$, then 
$$F(x,y) = 1 + (x+ y) +(x^2 +2xy+y^2)+(x+y)^3 + \cdots $$
and $f_{i,j} = \binom{i+j}{i}$. Now 
\begin{align*}
\frac{\d^j}{\d y^j} \frac{\d^i}{\d x^i} F(x,y) & = \frac{\d^j}{\d y^j} \frac{(-1)^i (-1)(-2)\cdots(-i)}{(1-x-y)^{i+1}} \\
	& = \frac{(i+j)!}{(1-x-y)^{i+j+1}}
\end{align*}
which when we evaluate at $(0,0)$ and divide by $i!j!$, we get $f_{i,j} = \frac{(i+j)!}{i!j!}$, as desired. 
\end{example}

By Theorem~\ref{THM:genfun} this means we can compute the rank $r+s$ surface tensor of $P$ {for fixed $s$ by taking all sets of $r$th order} partial derivatives of the generating function $MG_P(\bt)$.

\begin{example} 
Let $P$ be the quadrilateral given by $P = \conv\{x_1,\ldots, x_4\}$ with $x_1 = (0,0)$, $x_2=(1,0)$, $x_3=(1,1)$, and $x_4=(0,1)$. Since each facet is a $1$-dimensional simplex we have $Ad_{F_i}=1$ for $i=1,\ldots,4$, and furthermore $V_1(F_i) = 1$. So
\begin{align*}
MG_P(\bt) & = \sum_{r\geq 0} \frac{(r+1)!}{1!}\Phi^{r,s}_1(P)(\bt)\\
	& {= \frac{1}{s!\,\omega_{1+s}}\frac{\alpha_P^s(t_1,\ldots,t_4,u_1,\ldots, u_4)}{\LL_1\cdots\LL_4} }\\
	&  = \frac{1}{s!\,\omega_{1+s}}\frac{(1-t_1-t_2)(1-t_2)u_1^s + (1-t_2)u_2^s + (1-t_1)u_3^s + (1-t_1)(1-t_1-t_2)u_4^s}{1(1-t_1)(1-t_1-t_2)(1-t_2)} \\
	& = \frac{1}{s!\,\omega_{1+s}}\sum_{i=1}^4 \frac{u_{i}^s}{\LL_i\LL_{(i\bmod{4})+1}}.
\end{align*}

Now if we want the rank $1+s$ tensor we need the partial derivatives with respect to~$t_1$ and~$t_2$, which we calculate to be
\begin{align*}
\frac{\d}{\d t_j} MG_P(\bt) & = \frac{1}{s!\,\omega_{1+s}}\left(
\sum_{i=1}^4 \frac{-\left(\frac{\d}{\d t_j}\LL_i + \frac{\d}{\d t_j}\LL_{(i\bmod{4})+1}\right)u_i^s}{(\LL_i\LL_{(i\bmod{4})+1})^2} \right). 
\end{align*}
Each $\frac{\d}{\d t_j}\LL_i$ is simple to calculate and then evaluating at $\bt=0$ {(and trivially dividing by $1!0!$)}, we find 
$$\frac{\d}{\d t_1} MG_P(0) = \frac{1}{s!\omega_{1+s}}(u_1^s + 2u_2^s + u_3^s), \quad
\frac{\d}{\d t_2} MG_P(0) = \frac{1}{s!\omega_{1+s}}(u_2^s+2u_3^s+u_4^s). $$
We can verify our tensor calculation for $s=0$ as follows. The above derivatives will be the coefficient of each monomial of the $r=1$ term, $\frac{2!}{1!}\Phi^{1,0}_1(P)(\bt)$, {which we can further refine by writing as a sum of the coefficients of each $u_F^s$. Hence the $(1,0)$ and $(0,1)$ components of our tensor are as below} 
$$\Phi^{1,0}_1(P) = \frac{1}{2\,\omega_{1}}\left(
\begin{pmatrix} 1 \\0\end{pmatrix}u_1^0 + \begin{pmatrix} 2 \\1\end{pmatrix}u_2^0 + \begin{pmatrix} 1 \\2\end{pmatrix}u_3^0 + \begin{pmatrix} 0 \\1\end{pmatrix}u_4^0\right) = \begin{pmatrix} 1\\1 \end{pmatrix}.$$
Then calculating the Minkowski tensor from its definition, we have
\begin{align*}
\Phi^{1,0}_1(P) & = \frac{1}{1!\,\omega_{1}}\int_{\RR^2\times S^{1}}
\raisebox{1pt}{ \hspace{-15pt}$x^1$} \sum_{F\in\FF_{1}(P)}\int_{F}\int_{N(F,P)\cap S^{1}} \raisebox{1pt}{ \hspace{-35pt} $\mathbf{1}_{d(x,u)}(x,u)\HH^{0}(du)\HH^{1}(dx)$} \\
	& = \frac{1}{2}\int_{\RR^2\times S^1} x \sum_{F\in\FF_1(P)}\int_F \mathbf{1}_{d(x,u)}(x)\delta_{u_F} \HH^1(dx) \\
	& = \frac{1}{2}\sum_{F\in\FF_1(P)} \int_F x\, \HH^1(dx) \\
	& = \frac{1}{2} \int_{0}^1 \left[\begin{pmatrix} s \\ 0 \end{pmatrix} + \begin{pmatrix} 1 \\ s \end{pmatrix} + \begin{pmatrix} 1-s \\ 1 \end{pmatrix} + \begin{pmatrix} 0 \\ 1-s \end{pmatrix}\right] ds
\end{align*}
where the last equality comes from parametrizing each facet in variable $s$ which represents length (1-dimensional Hausdorff measure) along the facet. Now a simple integration gives the same result, $\begin{pmatrix}1\\1\end{pmatrix}$ as above. {To see the calculation of $\Phi^{2,0}_1(P)$ using Sage, see Appendix~\ref{sec:sage}.}
 
\end{example}

\subsection{Simplicial polytopes}
In fact, the above example is simply a special case of a simplicial polytope. In this case the adjoint of each facet is equal to $1$ so that the surface generating function has a particularly nice form. By deriving a formula for any $r$th order partial derivative of such a generating function, we get an explicit expression for Minkowski surface tensors of simplicial polytopes. 

Recall the generating function for our Minkowski tensors given in Theorem~\ref{THM:genfun} {in the simplicial case} is
$$MG_P(\bt) =  \frac{1}{s!\,\omega_{1+s}}\sum_{F\in\FF_{d-1}(P)} \frac{V_{d-1}(F)u_F^s}{\prod_{k\in F}\LL_k}. $$
Fix an integer $r\geq 1$. To simplify notation, we denote taking the derivative with respect to $t_i$ by $\d_i$ and given a sequence of indices $I = \{i_1,\ldots, i_r\}$ we denote by $\d_I$ the $r$th order partial derivative operator $\d_{i_r}\cdots\d_{i_1}$.
For each facet of $P$, we also write $\LL_F$ for the product $\prod_{k\in F}\LL_k$ in the denominator of each term of $MG_P(\bt)$.

\begin{prop} Let $P$ be a simplicial $d$-polytope, and $I=\{i_1,\ldots, i_r\}$ with each $i_j\in\{1,\ldots,d\}$. Then the $r$th order partial derivative $\d_I$ of the generating function for the surface tensors of $P$ is
\begin{align*}
{\frac{1}{s!\,\omega_{1+s}}\sum_{F\in\FF_{d-1}(P)}\hspace{-5pt}V_{d-1}(F)u_F^s \left[
\sum_{k=1}^{r} (-1)^{k}(k)!\frac{\LL_F^{r-k}}{\LL_F^{r+1}}\left(\sum_{\underset{I_1\dot\cup \cdots\dot\cup I_{k} = I}{\{I_1,\ldots, I_{k}\}}} \hspace{-10pt}\d_{I_1}\LL_F\cdots \d_{I_{k}}\LL_F \right)
\right],}
\end{align*}
{where the innermost sum is over all multiset partitions of $I$; that is, $I_1\cup\cdots\cup I_{k}=I$, $|I_1|+\cdots+|I_{k}| = |I|$, and changing the order of the sets $I_j$ doesn't change the partition}.
\label{PROP:deriv_formula}
\end{prop} 

\begin{proof} 
{\color{black}
       We do induction on the number of indices of $I=\left\{ i_{1},\cdots,i_{r} \right\}$.

       For $r=1$,  $I=\left\{ i \right\}$ ,
       \begin{equation*}
         \begin{split}
         \d_i MG_{P}(\bt)
         &=\frac{1}{s!w_{s+1}}\sum_{F\in \mathcal{F}_{d-1}(P)}\d_i \left(  \frac{V_{d-1}(F)u_{F}^{s}}{ \LL_{F}}\right)\\ & =\frac{1}{s!w_{s+1}}\sum_{F\in \mathcal{F}_{d-1}(P)}V_{d-1}(F)u_{F}^{s}\left[(-1)1!  \frac{\LL_{F}^{0}}{\LL_{F}^{2}}\left(  \d_{i}\LL_{F}\right)\right]
         \end{split}
       \end{equation*} so that the formula holds.
       
       Now we assume that the formula holds for $I=\left\{ i_{1},\cdots,i_{r} \right\}$, and then take the derivative with respect to $t_{i_{r+1}}$.
       \begin{equation*}
         \begin{split}           
           &\d_{i_{r+1}}\d_{I}MG_{P}(\bt)\\
          =&\frac{1}{s!w_{s+1}}\sum_{F\in \mathcal{F}_{d-1}(P)}V_{d-1}(F)u_{F}^{s} \d_{i_{r+1}}\left[ \sum_{k=1} ^{r}(-1)^{k}(k)!\frac{\LL_{F}^{r-k}}{\LL_{F}^{r+1}}\left( \sum_{I_{1}\cup\cdots\cup I_{k}=I} \d_{I_{1}}\LL_{F}\cdots\d_{I_{k}}\LL_{F}\right)\right].
           \end{split}
         \end{equation*}
        Now it suffices to show that we get the correct expression for each summand corresponding to some facet $F$. 
         Then,
         \begin{align*}
            \d_{i_{r+1}}\left[ \sum_{k=1} ^{r}(-1)^{k}(k)!\frac{\LL_{F}^{r-k}}{\LL_{F}^{r+1}}\left( \sum_{I_{1}\cup\cdots\cup I_{k}=I} \d_{I_{1}}\LL_{F}\cdots\d_{I_{k}}\LL_{F}\right)\right] \\
             &\hspace{-230pt} =\sum_{k=1} ^{r}(-1)^{k}(k)!\,\d_{i_{r+1}}\left(  \frac{\LL_{F}^{r-k}}{\LL_{F}^{r+1}}\left( \sum_{I_{1}\cup\cdots\cup I_{k}=I} \d_{I_{1}}\LL_{F}\cdots\d_{I_{k}}\LL_{F}\right)\right)\\
             &\hspace{-230pt} =\sum_{k=1} ^{r}(-1)^{k}(k)! \left[ \left(  \frac{\LL_{F}^{r+1}(r-k)\mathcal{L}_{F}^{r-k-1}\d_{i_{r+1}}\LL_{F}-\LL_{F}^{r-k}(r+1)\LL_{F}^{r}\d_{i_{r+1}}\LL_{F}}{\LL_{F}^{2r+2}}\right)\left( \sum_{I_{1}\cup\cdots\cup I_{k}=I} \d_{I_{1}}\LL_{F}\cdots\d_{I_{k}}\LL_{F}\right)\right.\\
             &\hspace{-230pt}\phantom{=}\left.+\frac{\LL_{F}^{r-k}}{\LL_{F}^{r+1}}\left( \d_{i_{r+1}}\sum_{I_{1}\cup\cdots\cup I_{k}=I} \d_{I_{1}}\LL_{F}\cdots\d_{I_{k}}\LL_{F}\right)\right]\\
             &\hspace{-230pt}=\sum_{k=1}^{r}(-1)^{k}(k)! \frac{(-k-1)\LL_{F}^{2r-k}\d_{i_{r+1}}\LL_{F} }{\LL_{F}^{2r+2}} \left( \sum_{I_{1}\cup\cdots\cup I_{k}=I} \d_{I_{1}}\LL_{F}\cdots\d_{I_{k}}\LL_{F}\right)\\
             &\hspace{-230pt}\phantom{=}+\sum_{k=1}^{r}(-1)^{k}(k)!\frac{\LL_{F}^{r-k}}{\LL_{F}^{r+1}}\left( \d_{i_{r+1}}\sum_{I_{1}\cup\cdots\cup I_{k}=I} \d_{I_{1}}\LL_{F}\cdots\d_{I_{k}}\LL_{F}\right)\\
             &\hspace{-230pt}=\sum_{k=1} ^{r}(-1)^{k+1}(k+1)!\frac{\LL_{F}^{r-k}}{\LL_{F}^{r+2}}\left[ \left( \d_{i_{r+1}}\LL_{F}\right)\left( \sum_{I_{1}\cup\cdots\cup I_{k}=I} \d_{I_{1}}\LL_{F}\cdots\d_{I_{k}}\LL_{F}\right)\right.\\
              &\hspace{-230pt}\phantom{=}+\sum_{k=1} ^{r}(-1)^{k}(k)!\frac{\LL_{F}^{r-k+1}}{\LL_{F}^{r+2}} \left( \sum_{I_{1}\cup\cdots\cup I_{k}=I} \d_{i_{r+1}}\left(  \d_{I_{1}}\LL_{F}\cdots\d_{I_{k}}\LL_{F}\right)\right) \\
             &\hspace{-230pt} =\sum_{k=1} ^{r}(-1)^{k+1}(k+1)!\frac{\LL_{F}^{r-k}}{\LL_{F}^{r+2}} \left( \sum_{I_{1}\cup\cdots\cup I_{k}\cup \left\{ i_{r+1} \right\}=I\cup\left\{ i_{r+1} \right\}} \d_{I_{1}}\LL_{F}\cdots\d_{I_{k}}\LL_F\d_{i_{r+1}}\LL_{F}\right)\\
             &\hspace{-230pt}\phantom{=} +\sum_{k=1} ^{r}(-1)^{k}(k)!\frac{\LL_{F}^{r-k+1}}{\LL_{F}^{r+2}} \left( \sum_{\underset{I_j' = I_j \mbox{ \tiny or } I_j'=I_j\cup\{i_{r+1}\}}{I'_{1}\cup\cdots\cup I'_{k}=I\cup\left\{ i_{r+1} \right\}}} \left(  \d_{I'_{1}}\LL_{F}\cdots\d_{I'_{k}}\LL_{F}\right)\right) \\
             &\hspace{-230pt} =\sum_{k=1} ^{r}(-1)^{k+1}(k+1)!\frac{\LL_{F}^{r-k}}{\LL_{F}^{r+2}} \left( \sum_{I_{1}\cup\cdots\cup I_{k}\cup \left\{ i_{r+1} \right\}=I\cup\left\{ i_{r+1} \right\}} \d_{I_{1}}\LL_{F}\cdots\d_{I_{k}}\LL_F\d_{i_{r+1}}\LL_{F}\right)\\
             &\hspace{-230pt}\phantom{=} +\sum_{k=0} ^{r-1}(-1)^{k+1}(k+1)!\frac{\LL_{F}^{r-k}}{\LL_{F}^{r+2}} \left( \sum_{I'_{1}\cup\cdots\cup I'_{k+1}=I\cup\left\{ i_{r+1} \right\}} \left(  \d_{I'_{1}}\LL_{F}\cdots\d_{I'_{k+1}}\LL_{F}\right)\right) \\
            &\hspace{-230pt}= \sum_{k=1} ^{r}(-1)^{k+1}(k+1)!\frac{\LL_{F}^{r-k}}{\LL_{F}^{r+2}}\left( \sum_{I_{1}\cup\cdots\cup I_{k+1}=I\cup\{i_{r+1}\}} \d_{I_{1}}\LL_{F}\cdots\d_{I_{k+1}}\LL_{F}\right) \\
          \end{align*}
         {where the last equality comes from the fact that a set partition of $I\cup\{i_{r+1}\}$ can have one of two forms: a partition of $I$ to which we add the extra set $\{i_{r+1}\}$, or a partition of $I$ where we add $i_{r+1}$ to one of the existing sets. This then} completes the induction.
         }
       \end{proof}

\begin{example}Let $P$ be a tetrahedron. Then its facets are indexed by all size three subsets of $\{1,2,3,4\}$, and so

\begin{align*}
MG_P(\bt) &= \frac{1}{s!\,\omega_{1+s}}\sum_{F\in \binom{[4]}{3}} \frac{V_{2}(F)u_{F}^s}{\LL_F} \\
\d_j MG_P(\bt) & = \frac{1}{s!\,\omega_{1+s}}\sum_{F\in \binom{[4]}{3}} {V_{2}(F)u_{F}^s} \frac{-\d_j \LL_F}{\LL_F^2} \\
\d_k\d_j MG_P(\bt) & = \frac{1}{s!\,\omega_{1+s}}\sum_{F} {V_{2}(F)u_{F}^s} \frac{\LL_F^2(-\d_k\d_j \LL_F) - (-\d_j\LL_F(2\LL_F\d_k \LL_F))}{\LL_F^4} \\
	& = \frac{1}{s!\,\omega_{1+s}}\sum_{F} {V_{2}(F)u_{F}^s} \frac{2\d_k\LL_F\d_j\LL_F - \LL_F\d_k\d_j\LL_F}{\LL_F^3} \\
\d_m\d_k\d_j MG_P(\bt) & = \frac{1}{s!\,\omega_{1+s}}\sum_F V_2(F)u_F^s \left( \frac{\LL_F^3(2\d_m\d_k\LL_F\d_j\LL_F + 2\d_m\d_j\LL_F\d_k\LL_F - \d_m\LL_F\d_k\d_j\LL_F - \LL_F\d_m\d_k\d_j \LL_F }{\LL_F^6}\right. \\
& \hspace{100pt} \left.\frac{- (2\d_k\LL_F\d_j\LL_F - \LL_F\d_k\d_j\LL_F)3\LL_F^2\d_m\LL_F}{\LL_F^6}\right) \\
& = \frac{1}{s!\,\omega_{1+s}}\sum_F V_2(F)u_F^s \left(\frac{2\LL_F(\d_m\d_k\LL_F\d_j\LL_F+\d_m\d_j\LL_F\d_k\LL_F + \d_k\d_j\LL_F\d_m\LL_F) }{\LL_F^4}\right.\\ 
& \hspace{100pt} \left.\frac{-6\d_m\LL_F\d_k\LL_F\d_j\LL_F- \LL_F^2\d_m\d_k\d_j\LL_F}{\LL_F^4}\right) \\
& = \frac{1}{s!\,\omega_{1+s}}\sum_F V_2(F)u_F^s
\left[
\sum_{i=1}^{3} (-1)^{i}(i)!\frac{\LL_F^{3-i}}{\LL_F^{4}}\left(\sum_{\underset{I_1\cup \cdots\cup I_{i} = \{j,k,m\}}{\{I_1,\ldots,I_{i}\}}} \hspace{-10pt}\d_{I_1}\LL_F\cdots \d_{I_{i}}\LL_F \right)
\right].
\end{align*}
\label{EX:tetrahedron}
\end{example}

Since each $\LL_F$ is a product of linear polynomials, it is not hard to calculate $\d_I\LL_F$. For $I = \{i_1,\ldots,i_r\}$ and $F$ with vertices indexed by $\{p_1,\ldots,p_{d}\}$, we have  
\begin{equation}
\d_I\LL_F = \begin{cases} 0 & \text{if } r > d \\
		\displaystyle\sum_{J=\{j_1,\ldots,j_r\}\subset\binom{[d]}{r}}\sum_{\sigma\in\frak{S}_{r}}\left(\d_{i_1}\LL_{p_{j_{\sigma(1)}}}\cdots \d_{i_r}\LL_{p_{j_{\sigma(r)}}}
		\prod_{k\notin J}\LL_k\right) & \text{else} \end{cases}
\label{EQ:Lpartial}\end{equation}
Then using the fact that $\d_j\LL_i = -x_{i,j}$ and $\LL_F(0,\cdots, 0) = 1$ for any $F$, we get the following. 
\begin{equation}
\d_I\LL_F(0,\ldots,0) = \begin{cases} 0 & \text{if } r > d \\
		\displaystyle\sum_{J=\{j_1,\ldots,j_r\}\subset\binom{[d]}{r}}\sum_{\sigma\in\frak{S}_{r}}(-1)^rx_{p_{j_{\sigma(1)}},i_1}\cdots x_{p_{j_{\sigma(r)}},i_r} & \text{else} \end{cases}
\label{EQ:Lpartialx}\end{equation}

The reader familiar with symmetric functions may now see that the above expression is reminiscent of the definition of an elementary symmetric function. We review that definition now and make precise the connection to the above derivatives. 

{
The elementary symmetric function $e_k(x_1,\ldots, x_m)$ of degree $k$ in $m$ variables is defined as follows
$$e_k(x_1,\ldots,x_m) = \sum_{1\leq j_1 < \cdots < j_k \leq m} x_{j_1}\cdots x_{j_k}.$$
So for example,
$$e_2(x_1,x_2,x_3) = x_1x_2+x_1x_3+x_2x_3.$$
By construction $e_k(x_1,\ldots,x_m)$ is invariant under the action of the symmetric group~$\frak{S}_m$, where $\sigma$ acts on $e_k$ by permuting the variables, $\sigma\cdot e_k(x_1,\ldots, x_m) = e_k(x_{\sigma(1)},\ldots, x_{\sigma(m)})$. This means that we can also write
$$(m-k)!k!\cdot e_k(x_1,\ldots,x_m) = \sum_{\sigma\in\frak{S}_m} x_{\sigma(1)}\cdots x_{\sigma(k)}.$$

Now we notice that we can rewrite  \eqref{EQ:Lpartialx} using the same notation to get 
$$\d_I \LL_F (0,\ldots,0) = \begin{cases} 0 & \text{if } r > d \\
		\displaystyle\frac{1}{(d-r)!}\sum_{\sigma\in\frak{S}_{d}} (-1)^r x_{p_{\sigma(1)},i_1} \cdots x_{p_{\sigma(r)},i_r} & \text{else}
 \end{cases}
$$
which is very close to the expression for the elementary symmetric function $e_r(x_1,\ldots,x_d)$ except that now we have double indexed variables. 
{\begin{defn}
Denote by $x_1,\ldots,x_d$ a set of variables $x_{1,1},\ldots, x_{1,n},\ldots, x_{d,1}\ldots,x_{d,n}$. Let $I = \{i_1,\ldots,i_k\}$ with each $i_j\in[n]$ and $k\leq d$. Then denote by $e_k^I(x_1,\ldots,x_d)$ the following doubly indexed ``elementary symmetric function''
$$(d-k)!\cdot e_k^I(x_1,\ldots,x_d) = \sum_{\sigma\in\frak{S}_d} x_{\sigma(1),i_1}\cdots x_{\sigma(k),i_k}.$$ 
\label{DEF:doub_elem_sym}\end{defn}}
\noindent{Notice that $e_k^I$ is also invariant under the action of $\mathfrak{S}_d$, {where $\mathfrak{S}_d$ acts by $\sigma\cdot e_k^I(x_1,\ldots,x_d) = e_k^I(x_{\sigma(1)},\ldots,x_{\sigma(d)}).$ Note that in what remains will always use $n=d$.}
\begin{prop} Let $F$ be a facet of simplicial polytope $P =\conv\{x_1,\ldots, x_m\}$ with $\LL_F = \prod_{k\in F}\LL_k$. Then for $I = \{i_1,\ldots,i_r\}$, the $r$th order partial derivative $\d_{i_r}\cdots \d_{i_1} \LL_F$ for $r\leq d$ evaluated at zero can be written as follows.
$$\d_I\LL_F(0,\ldots,0) =  (-1)^re_r^I(x_k : k\in F).$$
\label{PROP:partial_via_elemsymfn}
\end{prop}
}

To simpllify notation we define $e^I(F) := e_r^I(x_k : k\in F)$, where $r$ is implicitly given as $|I|$. We then have the following result which gives us the components of the Minkowski tensors. 
\begin{corollary} Let $I$ range over all possible choices {$I = \{i_1 \leq i_2\leq\cdots \leq i_r\}$} 
with each $i_j\in \{1,\ldots, d\}$. Denote by $a(I)_j$ the multiplicity of $j$ in $I$ for each $j\in[d]$. Then for a fixed $s\geq 0$ the rank $r+s$ Minkowski surface tensor of simplicial $d$-polytope $P=\conv\{x_1,\ldots,x_m\}$ is 
$$\Phi_{d-1}^{r,s}(P) = \frac{(d-1)!}{(r+d-1)!s!\,\omega_{1+s}}\sum_{I} \frac{c_I}{a(I)_1!\cdots a(I)_d!} e_{1}^{a(I)_1}\otimes\cdots\otimes e_d^{a(I)_d},$$ where the $c_I$ are given below. 
$$c_I := \sum_{F\in\FF_{d-1}(P)} V_{d-1}(F)u_F^s \left[
\sum_{k=1}^{r} (-1)^{k+r}k!\left(\sum_{\stackrel{I_1,\ldots,I_{k}}{\stackrel{I_1\cup \cdots\cup I_{k} = I}{|I_j|\leq d}}} \prod_{j=1}^{k} e^{I_j}(F) \right)\right]
$$
{where the innermost sum is over all multiset partitions of $I$; that is, $I_1\cup\cdots\cup I_{k}=I$, $|I_1|+\cdots+|I_{k}| = |I|$, and changing the order of the sets $I_j$ doesn't change the partition}.
\label{COR:simplicial_tensor_formula} 
\end{corollary}

{\color{black}
\begin{proof} Recall our generating function has the form
\begin{align*}
\sum_{r\geq 0} \frac{(r+d-1)!}{(d-1)!}\Phi_{d-1}^{r,s}(P)(\bt) & =
 \frac{1}{s!\,\omega_{1+s}}\sum_{F\in\FF_{d-1}(P)} \frac{V_{d-1}(F)\Ad_F(t_1,\ldots,t_d)u_F^s}{\prod_{k\in F}\LL_k} \\
 \end{align*}
 so that 
 $$\Phi_{d-1}^{r,s}(P) = \frac{(d-1)!}{(r+d-1)!}\sum_{\underset{j_1+\cdots+j_d=r}{j_1,\ldots,j_d\geq 0}}f_{j_1,\ldots,j_d} e_1^{j_1}\otimes\cdots\otimes e_d^{j_d},$$
 where $f_{j_1,\ldots,j_d}$ is the coefficient of the monomial $t_1^{j_1}\cdots t_d^{j_d}$ in the generating function. Then using Proposition~\ref{PROP:deriv_formula}, we have 
\begin{align*}f_{j_1,\ldots,j_d} & = \frac{1}{j_1!\cdots j_d!}\frac{\d^{j_1}}{\d t_{1}^{j_1}}\cdots \frac{\d^{j_d}}{\d t_{d}^{j_d}} MG_P\,\raisebox{-8pt}{\rule{0.5pt}{20pt}}_{(0,\ldots,0)} \\
 	& = \frac{1}{j_1!\cdots j_d!s!\,\omega_{1+s}}\sum_{F\in\FF_{d-1}(P)}\hspace{-5pt}V_{d-1}(F)u_F^s \left[
\sum_{k=1}^{r} (-1)^{k}(k)!\frac{\LL_F^{r-k}}{\LL_F^{r+1}}\left(\sum_{\underset{I_1\cup \cdots\cup I_{k} = I}{\{I_1,\ldots, I_{k}\}}} \hspace{-10pt}\d_{I_1}\LL_F\cdots \d_{I_{k}}\LL_F \right)
\right]\raisebox{-20pt}{\rule{0.5pt}{45pt}}_{(0,\ldots,0)}, \\
\end{align*}
where $I = \{\underbrace{1,\ldots,1}_{j_1 \mbox{\tiny times}},\ldots, \underbrace{d,\ldots,d}_{j_d \mbox{\tiny times}}\}$. Finally using Proposition~\ref{PROP:partial_via_elemsymfn} we get
\begin{multline*}
f_{j_1,\ldots,j_d} = \\
\frac{1}{j_1!\cdots j_d!s!\,\omega_{1+s}}\sum_{F\in\FF_{d-1}(P)}\hspace{-5pt}V_{d-1}(F)u_F^s \left[
\sum_{k=1}^{r} (-1)^{k}(k)!\left(\sum_{\underset{\underset{|I_\ell|\leq d}{I_1\cup \cdots\cup I_{k} = I}}{\{I_1,\ldots, I_{k}\}}} \hspace{-10pt}(-1)^{|I_1|}e^{I_1}(F)\cdots (-1)^{|I_{k}|}e^{I_{k}}(F) \right)
\right],
\end{multline*}
which, since $j_1+\cdots+j_d=r=|I|$ and $|I_1|+\cdots+|I_{k}| = |I|$, simplifies to the given expression. 
\end{proof}
}

\begin{example}Continuing with Example~\ref{EX:tetrahedron}, if we want the rank $3+s$ surface tensors then 
we use the $3$rd order partial derivatives we found earlier
$$\d_m\d_k\d_j MG_P(\bt) = \frac{1}{s!\,\omega_{1+s}}\sum_F V_2(F)u_F^s
\left[
\sum_{i=1}^{3} (-1)^{i}(i)!\frac{\LL_F^{3-i}}{\LL_F^{4}}\left(\sum_{\underset{I_1\cup \cdots\cup I_{i} = \{j,k,m\}}{\{I_1,\ldots,I_{i}\}}} \hspace{-10pt}\d_{I_1}\LL_F\cdots \d_{I_{i}}\LL_F \right)
\right],$$
which when we evaluate at $\mathbf{0}$ give the $c_I$ of Corollary~\ref{COR:simplicial_tensor_formula} for $I = \{j,k,m\}$.
$$c_{j,k,m} = \sum_{F\in\binom{[4]}{3}} V_{2}(F)u_F^s \left[
\sum_{i=1}^{3} (-1)^{i+3}i!\left(\sum_{\stackrel{I_1,\ldots,I_{i}}{\stackrel{I_1\cup \cdots\cup I_{i} = \{j,k,m\}}{|I_\ell|\leq 3}}} \prod_{\ell=1}^{i} e^{I_\ell}(F) \right)\right]
$$
Finally, we sum over all choices of $I = \{j\leq k\leq m\}$ with $j,k,m\in\{1,2,3\}$, to get 
$$
\Phi_{2}^{3,s}(P) = \frac{2!}{5!s!\,\omega_{1+s}}\sum_{I = \{j,k,m\}} \frac{c_{j,k,m}}{a(I)_1!\cdots a(I)_d!} e_{1}^{a(I)_1}\otimes\cdots\otimes e_3^{a(I)_3}.
$$
Then denoting a set $\{1,1,1\}$ by $111$, we have one term for each $I\in\{111,112,113,122,123,133,222$, $223,233,333\}$, and, for example, the $I = 113$ term of the tensor would be $\displaystyle\frac{2!}{5!s!\,\omega_{1+s}}\frac{c_{1,1,3}}{2!1!}e_1^2\otimes e_3$, where $c_{1,1,3}$ is given below\vspace{-3pt}
\begin{multline*} 
\sum_F V_2(F)u_F^s \left( e^{\{1,1,3\}}(F) - 2e^{\{1\}}(F)e^{\{1,3\}}(F) - 2e^{\{1\}}(F)e^{\{1,3\}}(F) -2e^{\{3\}}(F)e^{\{1,1\}}(F)\right. \\[-5pt]
\left.+ 6e^{\{1\}}(F)e^{\{1\}}(F)e^{\{3\}}(F) \right).
\end{multline*}
\end{example}

\begin{example} Consider an octahedron with realization $P = \conv\{\pm e_1,\pm e_2,\pm e_3\}$, where $e_i$ are the standard basis vectors in $\RR^3$. {Let $F$ range over the set of facets of $P$, $\{\conv(v_1e_1,v_2e_2,v_3e_3)\}_{v\in\{\pm1\}^3}$.} Then the rank $1$ surface tensor with $s=0$ would be 
$$\Phi_2^{1,0}(P) = \frac{2!}{3!\omega_1}\sum_{j=1}^3\left(\sum_{{F}} \frac{V_2(F)}{1!}\, {(-1)^{1+1}}{\left(v_1(e_1)_j+v_2(e_2)_j+v_3(e_3)_j\right)}\right)e_j = 0$$
since there are the same number of facets with vertex $+e_j$ as there are with vertex $-e_j$. 
{For the $s=0$ rank~$2$ surface tensor notice that for $F = \conv\{v_1e_1,v_2e_2,v_3e_3\}$, we get $e_1^{\{j\}}(F) = v_j$, and $e_2^{\{j,k\}}(F) = v_jv_k$, since $(e_i)_j = 0$ unless $i=j$. Thus}
\begin{align*}
\Phi_2^{2,0}(P) & = 
\frac{2!}{4!\omega_1}\sum_{k=1}^3\sum_{j=1}^3 \sum_{F} \hspace{-0pt} V_{2}(F) {\left[ (-1)^42!\, e^{\{j\}}(F)e^{\{k\}}(F) + (-1)^31!\, e_2^{\{j,k\}}(F) \right]e_j\otimes e_k} \\
	& = { \frac{1}{4!}\sum_{j=1}^3\sum_{k=1}^3 \sum_{v\in\{\pm 1\}^3} \frac{1}{2}\left(2!v_jv_k - 1! v_jv_k\right)e_j\otimes e_k
	}\\
	& = \frac{1}{4!}\sum_{k=1}^3\sum_{j=1}^3 \sum_{v\in\{\pm 1\}^3}{v_jv_k\,}
e_j\otimes e_k \\
	& = \left(\begin{array}{ccc} 1/3 & 0 & 0 \\ 0 & 1/3 & 0 \\ 0 & 0 & 1/3 \end{array}\right).
\end{align*}
We can check this against the integral definition of $\Phi_2^{r,0}(P)$ by parametrizing each facet appropriately. 
\begin{align*}
\Phi_2^{r,0}(P) & = \frac{1}{\omega_1} \int_{\RR^3\times S^2}x^r \sum_F \int_F\int_{N(F,P)} \mathbf{1}(x,u)\HH^0(du)\HH^2(dx) \\
	& = \frac{1}{2} \sum_F \int_F x^r \HH^2(dx) \\
	& = \frac{1}{2} \left[ \int_0^1\int_0^{1-x_2} \left(\begin{array}{c} x_1 \\ x_2 \\ 1-x_1-x_2 \end{array}\right)^r dx_1dx_2 + \cdots \right] \\
	& \stackrel{r=1}{=} \frac{1}{2}\sum_{v\in\{\pm1\}^3} \frac{1}{6}\left(\begin{array}{c} v_1 \\ v_2 \\ v_3 \end{array}\right) \\
	&  = 0.
\end{align*} 
\end{example}

\noindent In Appendix~\ref{sec:sage}, there is another example of computating a complete surface tensor using Sage.

\section{Questions} \label{sec:q}

As we have seen, the generating function approach is an effective way of both parametrizing and computing Minkowski surface tensors, especially for simplicial polytopes. On the other hand, there are still several directions in which further work could yield results.

\begin{enumerate}
\item {\em Beyond surface tensors.}

Having now seen generating functions for volume and surface Minkowski tensors, a natural next step would be to consider Minkowski tensors for $j \leq d-2$. For polytopes, this would correspond to considering lower dimensional faces, instead of just facets. 

One of the key factors that enabled us to give a nice expression for our surface tensor generating function is that in the surface tensor calculation the translation invariant component (the $u^s$ tensor) is integrated only against a dirac measure. 
However, the calculation of $\Phi^{r,s}_{d-2}$ could also be feasible since it only requires the integration in the $u$ variables over some $1$-dimensional arc of $S^{d-1}$. 
Moreover, for other tensors the purely translation covariant cases $\Phi_j^{r,0}$ are also promising as these tensors will also correspond to the calculation of moments of the $j$th dimensional facets of the polytope (up to some constant factors).

\item {\em Computational aspects.}

We have applied the method of Section~\ref{sec:simplicial} to give explicit formulations of the surface tensors of simplicial polytopes. Of course, these methods can be applied to the general case of a polytope with simplicial facets, but the calculations quickly become impractical by hand when the adjoint becomes more complicated. 
In general, in order to express the generating function $MG_{P}(\bt)$, and thus the surface tensors for a polytope $P$, the {computation of surface adjoints should be automated as in Algorithm~\ref{ALG:adj} (see Appendix~\ref{sec:App}) using some computer algebra package, such as SAGE \cite{sage} or Macaulay2 \cite{M2}, that can compute with polytopes. }

{With an efficient implementation of the surface adjoint and corresponding generating function, one could then test other interesting polytopes. For example, one might look at associahedra, permutahedra, or other zonotopes. It would be nice to find other families of polytopes for which the generating function has relatively simple form and the tensors can be written explicitly as we've done in the simplicial case. 
}

{There is also the issue of choosing the best method of coefficient extraction for the generating functions. For simplicial polytopes, we have used differentiation as it allows us  to deduce by hand an explicit expression for the coefficients. It could, however, be more efficient computationally or for different families of polytopes to use other generating function techniques. }

\item {\em Parametrizing varieties.}

We note in Proposition~\ref{PROP:surface_adjoint_vanish} that the surface adjoint vanishes on the union of the nonface arrangements of the facets. A natural question would be to further explore the vanishing of this polynomial. What are the properties of the vanishing loci of the surface adjoint polynomial? Does there exist an analog to Theorem~\ref{THM:adjoint_unique} for the surface adjoint?

Finally, as noted in \cite{momentpaper}, from the map assigning to each polytope its adjoint polynomial one can define a moduli space of Wachspress varieties. It is natural then to investigate what the map assigning to each polytope its surface adjoint, $P\mapsto \alpha^s_P$, represents and whether there is a similar interpretation as defining some moduli space of varieties. 

\end{enumerate}

\appendix
{\color{black}
\section{Adjoints algorithmically}\label{sec:App}
For reference, we present the calculation of the surface adjoint algorithmically. In Algorithm~\ref{ALG:adj} we compute a list of the adjoints of each facet of a given polytope $P$. In Algorithm~\ref{ALG:den}, we compute the denominator of the surface tensor generating function for $P$. This is simply a product of the appropriate linear forms coming from the vertices of $P$, and thus it can also be used to calculate the denominator of each facet summand of $MG_P(\bt)$ individually whenever that is advantageous computationally. Finally, in Algorithm~\ref{ALG:sadj} we present the calculation of the surface adjoint, which uses the adjoints found in Algorithm~\ref{ALG:adj}. This is then the numerator of the generating function for surface tensors, as described in Section~\ref{sec:genfn}, and together with the output of Algorithm~\ref{ALG:den} describes $MG_P(\bt)$.

We note that we found it particularly useful computationally to keep the numerator and denominator separate when wanting to calculate tensors via derivatives, as in Section~\ref{sec:simplicial}.

\begin{algorithm}[h]
\SetAlgoLined
\KwData{A polytope $P$}
\KwResult{List of $(F, AdF)$ where $F$ is a facet and $AdF$ is its adjoint }
$VP\leftarrow \left\{ v: v  \in \mathcal{F}_{0}(P) \right\}$\;
$d\leftarrow$ (ambient) dimension of $P$\;
$ \textit{Facets}\leftarrow\left\{ F: F\in \mathcal{F}_{d-1}(P)   \right\}$\;
$FacetAdj\leftarrow\left\{  \right\}$\;
$R\leftarrow\mathbb{Q}[t_{i}\text{ for i in} \left[ d \right]]$\;
\For{$F \in \textit{Facets}$}{
  $VF\leftarrow\left\{ v \in VP: v  \in F \right\}$\;
  $VolF\leftarrow V_{\dim(F)}(F)$\;
  \eIf{F is a simplex}{
    $AdF\leftarrow 1$\;
    Add $(F, AdF)$ to $FacetAdj$
  }
  {$\textit{triang}\leftarrow$ a triangulation of $VF$ as sets of vertices\;
    $AdF\leftarrow 0$\;
    \For{$T\in \textit{triang}$}{
      $\textit{nonface}\leftarrow\left\{ v: v \in VF\setminus  \textit{T}\right\}$\;
      $\sigma\leftarrow\conv(T)$\;
      $Vol\sigma\leftarrow Vol_{\dim(\sigma)}(\sigma)$\;
      $\textit{ad}\leftarrow\frac{Vol\sigma}{VolF}\prod_{v\in\textit{nonface}}(1-\sum_{i\in\left[ d \right]} v_{i}t_{i})$\;
      $AdF\leftarrow AdF+ad$
    }
    Add $(F,AdF)$ to $FacetAdj$}
}
\KwRet{$FacetAdj$}
\caption{Algorithm for computing the facet adjoints}\label{ALG:adj}
\end{algorithm}

\begin{algorithm}[!h]
\SetAlgoLined
\KwData{A polytope $P$}
\KwResult{Denominator of $MG_{p}(\bt)$ }
$VP\leftarrow\left\{ v: v  \in \mathcal{F}_{0}(P) \right\}$\;
$d\leftarrow$ (ambient) dimension of $P$\;
$R\leftarrow\mathbb{Q}[t_{i}\text{ for i in} \left[ d \right]]$\;
$\textit{Den}=\prod_{v\in VP}(1-\sum_{i\in \left[ d \right]}v_{i}t_{i})$\;
\KwRet{$\textit{Den}$}
\caption{Algorithm for computing the denominator of surface generating function}\label{ALG:den}
\end{algorithm}

\begin{algorithm}[!h]
\SetAlgoLined
\KwData{A polytope $P$}
\KwResult{Surface adjoint polynomial $Ad:=\alpha_P^s$}
$VP\leftarrow\left\{ v: v  \in \mathcal{F}_{0}(P) \right\}$\;
$d\leftarrow$ (ambient) dimension of $P$\;
$R\leftarrow\mathbb{Q}[t_{i}\text{ for i in} \left[ d \right]]$\;
$Ad\leftarrow 0$\;
$FacetAdj\leftarrow$ list of facet adjoints from Algorithm~\ref{ALG:adj}\;
\For{$(F,AdF) \in \textit{FacetAdj}$}{
$VF\leftarrow\left\{ v \in VP: v  \in F \right\}$\;
$Vol\leftarrow V_{\dim F}(F)$\;
$\textit{FactorF}\leftarrow\prod_{v\in VP\setminus VF}(1-\sum_{i\in \left[ d \right]}v_{i}t_{i})$\;
$Ad\leftarrow Ad + Vol\cdot AdF \cdot FactorF$\;
}
\KwRet{$\textit{Ad}$}
\caption{Algorithm for computing the surface adjoint polynomial}\label{ALG:sadj}
\end{algorithm}
}

\subsection{Using Sage}\label{sec:sage}
In what follows we calculate the surface tensor $\Phi^{2,0}_1(P)$ for polytope $P=\conv\{(0,0),(a,0),(a,b),(0,b)\}$ of Example~\ref{EX:quad}. We calculate this tensor both from the definition and using Corollary~\ref{COR:simplicial_tensor_formula} in Sage~\cite{sage}.

Recall from the derivation of \eqref{EQ:poly_surface_tensors} we have
$$\Phi^{r,s}_1(P) = \frac{1}{r!s!\,\omega_{1+s}}\sum_{F\in\FF_{1}(P)}\left( \int_Fx^r\HH^{1}(dx) \right)\otimes u_F^s, $$
which for our quadrilateral with $r=2, s=0$ becomes
\begin{align*}
\Phi^{2,0}_1(P) & = 	\frac{1}{2!0!\,\omega_{1}}\sum_{F\in\FF_{1}(P)}\left( \int_Fx^2\HH^{1}(dx) \right)\otimes u_F^0	\\
	& = \frac{1}{2!0!\,\omega_{1}}\left( \int_{(0,0)}^{(a,0)} x^2 \lambda(dx) + \int_{(a,0)}^{(a,b)} x^2 \lambda(dx) + \int_{(0,b)}^{(a,b)} x^2 \lambda(dx) +\int_{(0,0)}^{(0,b)} x^2 \lambda(dx) \right)
\end{align*}
which we calculate below.

\begin{lstlisting}
sage: #define facets of P
sage: a = var('a')
sage: b = var('b')
sage: assume(a>0)
sage: assume(b>0)
sage: F = [[[0,0],[a,0]],[[a,0],[a,b]],[[0,b],[a,b]],[[0,0],[0,b]]]
sage: r = 2, s = 0 #surface tensor rank 2+0
sage: omega1 = 2
sage: x = var('x')
sage: y = var('y')
sage: xx = matrix([x,y])
sage: T = xx.transpose()*xx #tensor x^2
sage: T
[x^2 x*y]
[x*y y^2]
sage: #integrate tensor T over each facet
sage: intF = [0,0,0,0]
sage: intF[0] = matrix(r, lambda i,j: integrate(T[i,j],x,0,a))
sage: intF[0] = intF[0].substitute({y:0}) 
sage: intF[1] = matrix(r, lambda i,j: integrate(T[i,j],y,0,b))
sage: intF[1] = intF[1].substitute({x:a}) 
sage: intF[2] = matrix(r, lambda i,j: integrate(T[i,j],x,0,a))
sage: intF[2] = intF[2].substitute({y:b}) 
sage: intF[3] = matrix(r, lambda i,j: integrate(T[i,j],y,0,b))
sage: intF[3] = intF[3].substitute({x:0}) 
sage: #each facet term is tensored with u_F^0, so remains as above
sage: Phi = 0
sage: for f in intF:
....:    Phi = Phi + f/(factorial(r)*factorial(s)*omega1)
sage: Phi
[  1/6*a^3 + 1/4*a^2*b 1/8*a^2*b + 1/8*a*b^2]
[1/8*a^2*b + 1/8*a*b^2   1/4*a*b^2 + 1/6*b^3]
\end{lstlisting} \vspace{5pt}

{Next we calculate the same tensor using Corollary~\ref{COR:simplicial_tensor_formula}. The function \texttt{elem(I,x)} calculates the doubly indexed elementary symmetric function $e_{|I|}^I(x)$ of Definition~\ref{DEF:doub_elem_sym} for some set of variables $x = \{x_1,\ldots, x_d\}$. The function \texttt{multisetPartitions(I,k,d)} calculates the collection of sets $I_1,\ldots, I_k$ that partition index set $I$ with each $|I_j|\leq d$, over which the we sum to get the $c_I$ of Corollary~\ref{COR:simplicial_tensor_formula}. Finally the function \texttt{c(I,F,d,s)} calculates the coefficient $c_I$ for index set $I=\{i_1\leq\cdots\leq i_r\}$ and set of facets $F$ with $d$ being the dimension of polytope $P$. }

{We use these functions to calculate the four components of the tensor $\Phi^{2,0}_1(P)$ coming from $I = \{i_1,i_2\} \in \{(1,1),(1,2),(2,2)\}$. Notice that by symmetry, the set $I = (1,2)$, which corresponds to the tensor $e_1^{1}\otimes e_2^{1}$, indexes a coefficient that will be divided into two entries of the tensor array.}
{
\begin{remark} Since Sage indexes from 0, all our index sets in the following are $-1$ from what is stated in the text.
\end{remark}
\begin{remark} Only the definition of \texttt{c(I,F,d,s)} below depends on our particular choice of $P$. This is only because Sage will not calculate the symbolic volume of a facet. However, give a concrete realization of some polytope, one can simply replace lines 60--62 with the calculation of the facet F as a Sage polyhedron.
\end{remark}
} 
\begin{lstlisting}
sage: def elem(I,x):
....:    #double indexed elementary symmetric function
....:    #indices I, variables x
....:    k = len(I) 
....:    d = len(x)
....:    e = 0
....:    for p in SymmetricGroup([i for i in range(d)]):
....:        e = e + product([x[p(j)][I[j]] for j in range(k)])
....:    return e/(factorial(d-k))
sage: def multisetPartitions(I,k,d):
....:    #partitions of the multiset I with k parts 
....:    #each part of size at most d
....:    n = len(I)
....:    P = []
....:    for p in SetPartitions(n,k):
....:        if max(list(len(s) for s in p)) <= d:
....:            P.append(list(list(I[j-1] for j in s) for s in p))
....:    return P
sage: def c(I,F,d,s):
....:    #coefficient c_I of tensor as in Corollary 5.6, 
....:    #using list of facets F for polytope in dimension d
....:    c = 0
....:    r = len(I)
....:    nF = len(F)
....:    u = list(var('u_%i' % i) for i in range(nF))
....:    for i in range(nF):
....:        #In this instance each facet is just a line segment
....:        V = list(F[i][1][j] - F[i][0][j] for j in range(2))
....:        V.remove(0)
....:        V = V[0]
....:        for k in range(1,r+1):
....:            cF = sum((-1)^(k+r)*factorial(k)*
....:                     product(elem(P[j],F[i]) for j in range(k)) 
....:                     for P in multisetPartitions(I,k,d))
....:            c = c + V*(u[i]^s)*cF
....:    return c
sage: comp = dict();
sage: for i1 in range(2):
....:    for i2 in range(i1,2):
....:        a = [[i1,i2].count(0),[i1,i2].count(1)]
....:        comp[i1,i2] = c([i1,i2],F,2,0)/(
....:            factorial(a[0])*factorial(a[1])*factorial(r+1)
....:            *factorial(s)*omega1)
sage: PhiAlt = matrix([[comp[0,0],comp[0,1]/2],[comp[0,1]/2,comp[1,1]]])
sage: PhiAlt
[  1/6*a^3 + 1/4*a^2*b 1/8*a^2*b + 1/8*a*b^2]
[1/8*a^2*b + 1/8*a*b^2   1/4*a*b^2 + 1/6*b^3]
sage: Phi == PhiAlt
True
\end{lstlisting} \vspace{5pt}

\noindent Finally, in line 80, we see that as expected the computations give the same result in both cases.

\newpage
\bibliographystyle{plain}
\bibliography{MT}

\Addresses

\end{document}